\documentclass[reqno,final,11pt]{amsart}
\usepackage{natbib}  
\usepackage{fancyhdr} 
\usepackage{color} 
\usepackage{geometry}
\usepackage[T1]{fontenc}
\usepackage{verbatim}
\RequirePackage[colorlinks,citecolor=blue,urlcolor=blue,linkcolor=blue]{hyperref}
\usepackage{amsmath,amsthm,amsfonts,setspace,tikz}
\usepackage{color}
\usepackage{fancybox}
\usepackage{epsfig}
\usepackage{subfig}
\usepackage{pdfsync}
\usepackage{enumerate}

\renewcommand{\cite}{\citet}

\theoremstyle{plain}
\newtheorem{theorem}{Theorem}[section]

\newtheorem{lemma}[theorem]{Lemma}

\theoremstyle{definition}

\theoremstyle{remark}
\newtheorem{remark}[theorem]{Remark}

\makeatletter \@addtoreset{equation}{section} \makeatother
\renewcommand\theequation{\thesection.\arabic{equation}}

      \def\PP{{\mathbb P}}

        \def\vv#1{{\boldsymbol #1}}

\def\wt#1{\widetilde{#1}}	

\def\topp#1{^{(#1)}}

        \def\calS{{\mathcal S}}

            \def\calC{{\mathcal C}}
              
                \def\calQ{{\mathcal Q}}
                  
                   \def\calM{{\mathcal M}}
                     \def\calX{{\mathcal X}}
                       \def\calY{{\mathcal Y}}
                         \def\calU{{\mathcal U}}
                           \def\calV{{\mathcal V}}
\makeatletter
\renewcommand*\env@matrix[1][\arraystretch]{%
  \edef\arraystretch{#1}%
  \hskip -\arraycolsep
  \let\@ifnextchar\new@ifnextchar
  \array{*\c@MaxMatrixCols c}}
\makeatother

\usepackage{stmaryrd}

\numberwithin{equation}{section}

\newcommand{\RR}{\mathbb{R}}

\newcommand{\ZZ}{\mathbb{Z}}
\newcommand{\NN}{\mathbb{N}}

\def\topp#1{^{(#1)}}

\newcommand{\C}{  \mathsf c }
\newcommand{\A}{  \mathsf a}
\newcommand{\B}{ \mathsf b}
\newcommand{\D}{  \mathsf d}

  \newcommand{\Dq}{   {\mathcal D}_q  }

\title[Stationary distribution of open ASEP]{Stationary distribution of open asymmetric simple exclusion processes on an interval as a marginal of a two-layer ensemble}

\author{W{\l}odek Bryc}
\address
{
Department of Mathematical Sciences\\
University of Cincinnati\\
2815 Commons Way\\
Cincinnati, OH, 45221-0025, USA.
}
\email{wlodek.bryc@gmail.com}



 \usepackage{framed,pgf}
\newcounter{oldeq}
\newcounter{usesofarxiv}
 \newcommand{\arxiv}[1]{
\setcounter{oldeq}{\value{equation}}
 \addtocounter{usesofarxiv}{1}
 \setcounter{equation}{0}
\def\theoldeq{\theequation}
\def\theequation{x-\arabic{usesofarxiv}.\arabic{equation}}
\def\theequation{\arabic{section}.\arabic{usesofarxiv}.\arabic{equation}}
\def\theequation{\thesection.\arabic{usesofarxiv}.\arabic{equation}}
  \colorlet{shadecolor}{gray!10}
{\footnotesize
\begin{shaded}#1
\end{shaded}
   \setcounter{equation}{\value{oldeq}}
\numberwithin{equation}{section}
}}

\begin{document}

\begin{abstract}
We investigate the asymmetric simple exclusion process (ASEP) on an interval with open boundaries. We provide a representation for its stationary distribution as a marginal of the top layer of a two-layer ensemble under Liggett's condition. The representation is valid in the fan region and in the shock region, extending the representation previously obtained in \cite{Bryc-Zatitskii-2024} to ASEP. We also give a recursion for the two-layer weight function.

\end{abstract}

\maketitle
\arxiv{This is an expanded version of the paper. }
\section{Introduction}
\subsection{Asymmetric simple exclusion process}

The asymmetric simple exclusion process (ASEP)  is an interacting particle system on $\ZZ$, $\NN$ or on a finite set  such that each site can have at most one particle, the particles can only move to the two nearest neighbors and are more likely to move to the right.
 The variant in which we are interested, the asymmetric simple exclusion process  with open boundaries, is a continuous-time finite-state Markov process that models the movement of particles along the   sites  $\{1,\dots,L\}$ where particles can leave or enter the segment $\{1,\dots,L\}$ at the boundaries. The process was first introduced in \cite{Macdonald1968kinetics} to model the concurrent progressive movement of multiple non-overlapping ribosomes of moderate size $\ell$ on a lengthy mRNA template.  For $\ell=1$, the non-overlapping  becomes exclusion condition and  the resulting process is now called the open ASEP.
There are 5 parameters that describe the asymmetry of jumps and   boundary behavior of ASEP: particles can move to the nearest site to the
right with the rate $1$ and to the left with the rate $0\leq q<1$, provided the
target site is empty.
 In addition, particles enter at the left side boundary location $1$ at rate
$\alpha>0$ and leave it at rate $\gamma\geq0$. Particles can exit at the right side boundary location $L$ at rate $\beta > 0$ or arrive there at rate $\delta\geq 0$.  This informal description is illustrated in Fig. \ref{Fig1}.
For a formal description of the infinitesimal generator of this Markov process (under the so called Liggett's condition \eqref{Liggets-greek}), we refer  e.g. to \cite[Section 3]{Liggett-1975} or \cite[Chapter 3]{Liggett-1999Stoch-Inter}.

The stationary measure for open ASEP is often called a steady state  as  there is  a net flux of particles flowing through the system. This stationary measure is of considerable interest and has been studied for a long time, starting with \cite{Macdonald1968kinetics}   and then \cite{Liggett-1975},  \cite{derrida93exact}.
 In particular, the celebrated matrix model \cite{derrida93exact} is an indispensable tool to study properties and build useful representations of the stationary measure.
  Numerous references are available: we mention here \cite{brak2006combinatorial,Bryc-Swieca-2018,Bryc-Wesolowski-2015-asep,Corteel-Williams-2011,
  Corteel-Williams-2013-erratum,enaud2004large,Uchiyama-Sasamoto-Wadati-2004,wang2023askey}. For more references, we refer to \cite{Bertini-2007,derrida07nonequilibrium,Liggett-1999Stoch-Inter}.

In this paper, we establish a representation for the stationary distribution of ASEP as a marginal of the top layer of a two-layer ensemble. This extends the two-line representation  \cite[Theorem 1]{Bryc-Zatitskii-2024} to $q>0$ and \cite[Section 2.3]{Barraquand2023Motzkin} beyond the fan region. Our main contribution is a semi-explicit rational function expression for the joint law of the two layers, and a recursion for the corresponding weight function. Other authors considered two-layer representations that are different from ours.    For $q=0$, Ref. \cite[Section 3.2]{duchi2005combinatorial} represents the stationary measure as a marginal of a "two-layer" ensemble, which is a stationary law for a Markov evolution on two-layers that we do not have for our representation.
Another two layer representation  implicit in \cite[Section 5.2]{nestoridi2023approximating} overlaps with ours for a special choice of parameters in both models.

\begin{figure}[tbh]
 \begin{tikzpicture}[scale=.8]
\draw [fill=black, ultra thick] (.5,1) circle [radius=0.2];
  \draw [ultra thick] (1.5,1) circle [radius=0.2];
\draw [fill=black, ultra thick] (2.5,1) circle [radius=0.2];
  \draw [ultra thick] (5,1) circle [radius=0.2];
   \draw [fill=black, ultra thick] (6,1) circle [radius=0.2];

   \node [above] at (-1.5,1.) {left};
   \node [below] at (-1.5,1.) {reservoir};
      \draw[-,dashed,blue] (-2.5,-1) to (-2.5,3);
    \draw[-,dashed,blue] (-0.5,-1) to (-0.5,3);
        \draw[-,dashed,blue] (-2.5,3) to (-0.5,3);
             \draw[-,dashed,blue] (-2.5,-1) to (-0.5,-1);
      \node [above] at (12.5,1.) {right };
       \node [below] at (12.5,1.) {reservoir};
      \draw[-,dashed,blue] (11.5,-1) to (11.5,3);
       \draw[-,dashed,blue] (11.5,3) to (13.5,3);
       \draw[-,dashed,blue] (11.5,-1) to (13.5,-1);
         \draw[-,dashed,blue] (13.5,-1) to (13.5,3);
    \draw [ultra thick] (7,1) circle [radius=0.2];
      \draw [fill=black, ultra thick] (9.5,1) circle [radius=0.2];
   \draw [ultra thick] (10.5,1) circle [radius=0.2];
     \draw[->,dashed] (-1,2.3) to [out=-20,in=135] (.5,1.5);
   \node [above right] at (-.2,2) {$\alpha$};
   \draw [fill=black, ultra thick] (-1.3,2.5) circle [radius=0.2];
     \draw[->,dashed] (10.5,1.5) to [out=45,in=200] (12,2.3);
     \node [above left] at (11.2,2) {$\beta$};
           \draw [fill=black, ultra thick] (12.3,2.3) circle [radius=0.2];
            \node  at (8.25,1) {$\cdots$};  \node  at (3.75%
,1) {$\cdots$};
      \node [above] at (6.5,1.8) {$1$};
      \draw[->,thick] (6.1,1.5) to [out=45,in=135] (7,1.5);
        \node [above] at (5.5,1.8) {$q$};
            \draw[<-] (5,1.5) to [out=45,in=135] (5.9,1.5);
                 \node [above] at (10,1.8) {$1$};
                \draw[->,thick] (9.6,1.5) to [out=45,in=135] (10.4,1.5);
               \node [above] at (9,1.8) {$q$};
                 \draw[<-] (8.4,1.5) to [out=45,in=135] (9.4,1.5);
       \draw[<-] (-1,-.3) to [out=0,in=-135] (.5,0.6);
   \node [below right] at (-.2,0) {$\gamma$};
     \draw [fill=black, ultra thick] (-1.3,-.2) circle [radius=0.2];
    \node [above] at (0.5,0) {$1$};
    \node [above] at (1.5,0) {$2$};
   \node [above] at (2.5,0) {$3$};
     \node [above] at (3.75,0) {$\cdots$};
   \node [above] at (5,0) {$k-1$};
       \node [above] at (7,0) {$k+1$};
        \node [above] at (8.25,0) {$\cdots$};
          \node [above] at (9.5,0) {$L-1$};
        \draw[<-] (10.5,.6) to [out=-45,in=180] (12,-.3);
   \node [below left] at (11.2,0) {$\delta$};
      \draw [fill=black, ultra thick] (12.3,-.2) circle [radius=0.2];
\end{tikzpicture}
\caption{Transition rates of the open ASEP with parameters  $0\leq q<1$, $\alpha,\beta>0$, $\gamma,\delta\ge 0$. \label{Fig1}}
\end{figure}

 We use the standard re-parameterizations of ASEP by parameters $\A,\B,\C,\D$ such that
\begin{equation}
  \label{Greek2ASEP}
  \alpha=\tfrac{1-q}{(1+\A)(1+\C)}, \quad \beta=\tfrac{1-q}{(1+\B)(1+\D)}, \quad \gamma=-\A \C \alpha,\quad \delta=-\B \D \beta,
\end{equation}
where  we choose $\A=\kappa_+(\alpha,\gamma)$, $\B=\kappa_+(\beta,\delta)$ to be nonnegative,  and then $\C=\kappa_-(\alpha,\gamma),\D=\kappa_-(\beta,\delta)$ are in $(-1,0]$.
The  notation
$$\kappa_{\pm}(u,v)=\frac{1-q-u+v\pm\sqrt{(1-q-u+v)^2+4u v}}{2u} $$
for the solutions of the pair of uncoupled quadratic equations \eqref{Greek2ASEP} is now standard; see \cite[(74)]{essler1996representations} \cite[(4.2)]{sandow1994partially} or \cite[(1.4)]{corwin-knizel2021stationary}. Parameters $\A,\B,\C,\D$ determine the related family of Askey-Wilson polynomials \cite{Bryc-Wesolowski-2015-asep,Uchiyama-Sasamoto-Wadati-2004,wang2023askey}. However, this relation does not play a direct role here.

\subsubsection{Liggett's condition}
In parametrization \eqref{Greek2ASEP}, the celebrated Liggett's condition
\begin{equation}
  \label{Liggets-greek}
  \gamma=q(1-\alpha) \mbox{ and  }\delta=q(1-\beta)
\end{equation} becomes  $\C=\D=-q$ and
\eqref{Greek2ASEP} simplifies to
\begin{equation*}\label{AB=..}
 \A=\frac{1-\alpha}{\alpha},\quad \B =\frac{1-\beta}{\beta}.
\end{equation*}
In our results we assume Liggett's condition and $\alpha,\beta\leq 1$. Thus,  in a slight extension of \cite{Bryc-Zatitskii-2024}, we consider $\A,\B\geq 0$.

As usual, we identify the particle configurations with the sequences of $0$  and $1$ that mark the vacant and occupied sites. The stationary measure of ASEP is then a discrete probability measure $\mu$ on $\Omega_L:=\{0,1\}^L$.
With some abuse of notation,  we write $\mu(\vv \tau)$ for the value assigned to the sequence $\vv \tau=(\tau_1,\dots,\tau_L)\in\Omega_L$; that is, we treat $\mu$ as a function defined on the union $\bigcup_{L}\Omega_L$, with the parameter $L$ implicitly determined by the length of the sequence $\vv \tau$.

\subsection{Notation}
 We use boldface notation, such as $\vv \sigma$, for sequences and standard font for the corresponding values, as in $\sigma_0,\dots,\sigma_r$.

The $q$-Pochhammer symbol is $ (a;q)_n=\prod_{k=0}^{n-1} (1-a q^k)$ and the $q$-number is $[n]_{q} :=1+q+\dots +q^{n-1}$.  For $q\ne 1$ the $q$-differentiation operator  (Jackson derivative) is
\begin{equation}
    \label{Dq-def} (\Dq f)(z) := \frac{f(z)-f(q z)}{z(1-q)} \; \mbox{ for } z\ne 0.
\end{equation}
(This standard notation appears, for example, in \cite[Chapter 1]{Koekoek} or \cite{gasper2004basic}.) We write $\Dq z$ for the operator that assigns  a function $\Dq[z f(z)]$ to a function $f$.

 A composition $\vv \sigma=(\sigma_0,\sigma_1,\dots,\sigma_r)$ of nonnegative integer $L+1$ into $r+1$ parts is a sequence of $r+1$ strictly positive integers such that
$$L+1=\sigma_0+\dots+\sigma_r.$$
Observe that the sequence of elements in the composition is significant, distinguishing it from the notion of a partition. For example, the partition \(\lambda = \{1,2,2\}=\langle1,2^2\rangle\) of 5 can result in three different compositions: \((1,2,2)\), \((2,1,2)\), or \((2,2,1)\).
(Refer to, for instance, \cite[page 14]{Stanley1997v1} or \cite[Definition I.9]{Flajolet2009}).

With a composition $\vv \sigma$ of $L+1$ we shall associate a polynomial $ w_{\vv \sigma}$ in variable $z$ which for $z \ne 0,1,1/q,1/q^2,\dots$ is given by
\begin{equation}
  \label{w2Dq}
 w_{\vv \sigma}(z)=(z;q)_{L+2}\prod_{j=0}^r \left( (\Dq z)^{\sigma_j-1}\Dq\right) [\tfrac{1}{1-z}].
\end{equation}
A priori, this is a rational function, but in Lemma \ref{lem1} we show that this is indeed a polynomial in variable $z$, so the expression $w_\vv\sigma(z)$ is defined for all real $z$.
 We also show that the coefficients are nonnegative  and that for $|z|<1$  there is an alternative expression
 \begin{equation}\label{w-def}
 w_{\vv \sigma}(z)=(z;q)_{L+2}\sum_{n=0}^\infty z^n \prod_{j=0}^r ([n+j+1]_q)^{\sigma_j}.
  \end{equation}
  In particular, $w_{\vv \sigma}(0)=\prod_{j=0}^r [j+1]_q^{\sigma_j}$, so $w_{\vv \sigma}(z)>0$ for $z\geq 0$.
  We also see that for all $\vv \sigma$ and $z$ we have $ w_{\vv \sigma}(z)=1$ for $q=0$, as then $[n+j+1]_q=1$.
 \arxiv{The extension of expression \eqref{w2Dq} to all real $z$ is not explicit. For example,  to see that $w_\sigma(0)=\prod_{j=0}^r [j+1]_q^{\sigma_j}$ we use \eqref{w-def}.

 Another value of interest is $w_{\vv\sigma}(1)=[L+1]_q!$ for all $\vv\sigma$, which is obtained by taking the limit of $w_{\vv \sigma}(z)$ as $z\to 1^-$. This argument relies on calculation $w_{(1,1,\dotsm,1)}(z)= [L+1]_q!$, which follows from \eqref{w2Dq} by \eqref{Dq-prop}, and then on two bounds
 $$
 w_{(L+1)}(z)\leq w_{\vv\sigma}(z)\leq w_{(1,1,\dots,1)}(z), \quad 0\leq z<1,
 $$
 and
 $$
 w_{(1,1,\dots,1)}(z)-w_{(L+1)}(z)\leq \frac{L q}{(1-q)^{L+1}}(1-z),\quad  0\leq z<1,
 $$
 which follow from \eqref{w-def}.
}
\subsection{The two layer ensemble}
We now introduce the two-layer ensemble, which is defined by a nonnegative weight function $Q$ on $\Omega_L\times\Omega_L$, normalized to give a probability measure.
For $(\vv \tau,\vv\xi)\in \Omega_L\times\Omega_L$ we write the weight function
$
Q \left(\begin{matrix}\vv \tau \\ \vv \xi \end{matrix}\right)
$ as a function of the top layer $\vv \tau$ and the bottom layer $\vv \xi$, as drawn on the right side of Fig. \ref{Fig:2L}. (This two-layer notation is vital for the statement of Theorem \ref{thm2}.  In this notation, $\tau_i=1$ if the $i$-th location in the top layer is occupied, which is marked by the black disk in Fig. \ref{Fig:2L}, and $\xi_j=1$ if the $j$-th location in the bottom layer is occupied.) As previously we do not write parameter $L$ explicitly, so we treat $Q$ as a mapping defined on $\bigcup_L\Omega_L\times\Omega_L$.

Our expression for $Q$ is based  on a  {\em random walk path}   $\vv \gamma=\vv \gamma(\vv \tau,\vv \xi)$, defined as the difference of partial sums in the top and the bottom layers:
\begin{equation}\label{gamma:def}
 \vv \gamma=(\gamma_0,\dots,\gamma_L)=\left(0,\tau_1-\xi_1, \dots, \sum_{j=1}^L (\tau_j-\xi_j)\right).
\end{equation}   We will draw this path as a piecewise linear curve, as indicated at the top of Fig. \ref{Fig:2L}. We note that our terminology here differs slightly from \cite[Section 2.5]{derrida04asymmetric} and \cite[Section 2.3]{Barraquand2023Motzkin} as  our random walk paths start at zero and can be negative.

Our formula  depends on the end point $\gamma_L$ of the path, its minimum
$$\min(\vv \gamma):=\min \left\{\sum_{j=1}^k(\tau_j-\xi_j): {0\leq k\leq L}\right\},$$ and on the induced  composition  $\vv \sigma=\sigma(\vv\gamma)=(\sigma_0,\dots,\sigma_r)$ of $L+1$ into $r+1$ parts with $r=\max (\vv \gamma)-\min(\vv \gamma)$, and
\begin{equation}\label{def:sigma}
 \sigma_i=\#\left\{j: \gamma_j=\min(\vv \gamma)+i\right\}, \quad i=0,\dots,r.
\end{equation}
There are exactly $3^L$ possible paths as $\vv \gamma$  has
possible increments $+1,0,-1$ in each step. So the mapping $\vv \gamma\mapsto \vv \sigma=\sigma(\vv \gamma)$ maps the set of $3^L$ paths $\vv \gamma$ onto the set of $2^L$ compositions $\vv \sigma$.

\arxiv{Every composition $\vv \sigma$ of $L+1$ arises from at least one
path
$\vv \gamma$. For example, every composition $\vv\sigma$ corresponds to a unique nondecreasing  path given by $$\gamma_0=\dots=\gamma_{\sigma_0-1}=0, \gamma_{\sigma_0}=\dots=\gamma_{\sigma_0+\sigma_1-1}=1, \dots,
\gamma_{\sigma_0+\sigma_1+\sigma_{r-1}}=\dots=\gamma_{\sigma_0+\sigma_1+\sigma_{r-1}+\sigma_r-1}=r.$$
We also remark that $\vv \sigma$ can be obtained by relabeling the non-zero entries of the sequence $(k_j)_{j=-L,\dots,L}$ that describes the empirical measure of a path,
$$\sum_{j=0}^L \delta_{\gamma_j}=\sum_{j=-L}^L k_j \delta_j.$$
That is, $\sigma_i=k_{\min(\vv\gamma)+i}$, $i=0,\dots,r$.
}

With the above notation, for $\A,\B\geq 0$, we define the two-layer weight function by
\begin{equation}
  \label{Q-def}
  Q \left(\begin{matrix}\vv \tau \\ \vv \xi \end{matrix}\right)=\frac{\B^{\gamma_L}}
  {(\A\B)^{\min(\vv \gamma)}} w_{\sigma(\vv \gamma)}(\A\B),
\end{equation}
  where $\vv \gamma=\gamma(\vv\tau,\vv\xi)$   is given by \eqref{gamma:def}. Since $Q$ depends only on $\vv \gamma$ and the function $(\vv\tau,\vv\xi)\mapsto \gamma(\vv\tau,\vv\xi)$ is injective, one could write the weight as a function of $(\vv \tau,\vv\gamma)$. This would lead to an equivalent and perhaps more intuitive form of Theorem \ref{thm1},  which  we discuss in Section \ref{sec:2ndproof}.  However, it would be difficult to state Theorem \ref{thm2}  without the two-layer notation.

Recall that $\min (\vv \gamma)\in\ZZ_{\leq 0}$, so by Lemma \ref{lem1},  expression \eqref{Q-def} defines $Q$ as a polynomial in $\A,\B$ with nonnegative coefficients. Since $ w_{\vv \sigma}(\A\B)\geq w_{\vv \sigma}(0)>0$, we see that $Q>0$ if $\A,\B>0$ and that $Q$ is not identically $0$ even if $\A=0$ and $\B=0$.
The two-layer ensemble
\begin{equation}
  \label{PTL1}
  P_{\rm TL}\left(\begin{matrix}\vv \tau \\ \vv \xi \end{matrix}\right)=\frac{1}{Z} Q \left(\begin{matrix}\vv \tau \\ \vv \xi \end{matrix}\right), \quad Z=Z_L(\A,\B,q)=\sum_{\vv \tau,\vv \xi \in \Omega_L} Q \left(\begin{matrix}\vv \tau \\ \vv \xi \end{matrix}\right),
\end{equation}
is a probability measure on $\Omega_L\times\Omega_L$ obtained by normalizing $Q$.  The normalization (partition function) $Z_L(\A,\B,q)$ is a nonzero polynomial in variables $\A,\B$ with nonnegative coefficients.

Our main result represents stationary measure of ASEP as the top marginal of the two layer ensemble.
The  formula covers both  the shock and the fan regions.
\begin{theorem}
  \label{thm1}
  If  $0<\alpha,\beta\leq 1$ and Liggett's condition \eqref{Liggets-greek} holds, then the invariant measure $\mu$ of ASEP is a marginal of the top layer of the two layer ensemble,
  \begin{equation}
    \label{TL2mu}
      \mu(\vv \tau)=\sum_{\vv \xi \in \Omega_L} P_{\rm TL}\left(\begin{matrix}\vv \tau \\ \vv \xi \end{matrix}\right), \quad \vv \tau\in\Omega_L.
  \end{equation}
\end{theorem}
We note that since $ w_{\vv \sigma}(z)=1$ for $q=0$, the above expression generalizes to $q>0$ the two-line ensemble representation in \cite{Bryc-Zatitskii-2024}. %

For the fan region $\A\B<1$, Theorem \ref{thm1} can be obtained quite directly from \cite[Section 2.3]{Barraquand2023Motzkin}.
We present this argument in Section \ref{sec:2ndproof}.
However, it is of interest to deduce the result from the recursion for the two-layer weight function $Q$ which leads to the so-called {\em basic weight equations}.
According to \cite[Theorem 1]{brak2006combinatorial}, the stationary measure $\mu$ of ASEP is
\begin{equation}
  \label{P:ASEP}
  \mu(\vv \tau)=\frac{1}{\mathsf{Z}_{L}} \Phi(\vv \tau), \quad \vv \tau\in \Omega_{ L},
\end{equation}
where
$$Z_L=\sum_{\vv \tau\in \Omega_{ L}}  \Phi(\vv \tau)$$ is the normalization constant and the {\em basic weight function} $\Phi:\bigcup\Omega_L\to \RR$ satisfies four  {\em basic weight equations}. When
specified to Liggett's condition and written in our parametrization, these  basic weight equations are
\begin{eqnarray}
\label{pini}    \Phi(\emptyset)&=&1 \\
\label{pa}    \Phi(0,\vv\tau) -q\, \A\, \Phi(1,\vv\tau)&=&(1+\A)\Phi(\vv\tau) \\
\label{pb}    \Phi(\vv\tau, 1)-q\,\B\, \Phi(\vv\tau,0) &=& (1+\B) \Phi(\vv \tau) \\
\label{pbulk}    \Phi(\vv\tau_1, 1,0,\vv\tau_2)-q \, \Phi(\vv\tau_1,0,1,\vv\tau_2) &=&
 \Phi(\vv\tau_1, 0, \vv\tau_2)+ \Phi(\vv\tau_1, 1,\vv\tau_2) .
\end{eqnarray}
It is known that if these equations have a solution which is not identically zero and the values of $\Phi$ have the same sign for a given $L$, then  \eqref{P:ASEP} gives the stationary measure of ASEP. It is also known that in the singular case where $\A\B\C\D q^N=1$ for some $N=1,2,\dots$, (which under Liggett's condition becomes $\A\B q^{N+2}=1$) the solution is zero for  $L\geq N+1$, see  \cite[Appendix A]{Mallick-Sandow-1997}   and \cite[Theorem 1]{Bryc-Swieca-2018}.
According to \cite[Remark 2]{Bryc-Swieca-2018}, in the nonsingular case, all solutions have the same sign, so \eqref{P:ASEP} still holds, with possibly negative normalization $Z_L$.
In our  proof of Theorem \ref{thm1} we give an explicit nonzero expression \eqref{Phi-sol} for $\Phi$ in the nonsingular case, so it is clear that the solution exists and that normalization yields positive probability measure $\mu$. %

For $\A \B \not\in \{1,1/q,1/q^2, \dots\}$, and $\vv\tau,\vv\xi\in\Omega_L$, $L=0,1,\dots,$ consider
\begin{equation}
  \label{wt Q}
  \wt Q \left(\begin{matrix}\vv \tau \\ \vv \xi \end{matrix}\right):=
\frac{(\A\B;q)_2}{(\A\B;q)_{L+2}} Q \left(\begin{matrix}\vv \tau \\ \vv \xi \end{matrix}\right)
\end{equation}
 with  $\wt Q(\emptyset)=1$ for $L=0$. Since normalization cancels out any common (nonzero) factor, it is clear that \begin{equation}\label{wtW2PTL}
    P_{\rm TL} = \frac{1}{\wt Z} \wt Q.
  \end{equation}
However, note that while $Q\geq 0$, the sign of $\wt Q$ may vary with $L$ for $\A\B q^2>1$. This aligns with \cite[Remark 2]{Bryc-Swieca-2018} and is essential for deriving the basic weight equations for $\Phi$.

Our second main result is the following version of the basic weight equations for the two-layer ensemble.
\begin{theorem}\label{thm2} For $\A\B \not\in \{1,1/q,1/q^2,\dots\}$,  the two-layer weight function $\wt Q$
satisfies the following boundary and bulk equations:
\begin{equation}
  \label{Qa}
  \wt Q \left(\begin{matrix}0 &\vv \tau \\ \xi' &\vv \xi \end{matrix}\right)
  - q\, \A \wt Q \left(\begin{matrix}1 &\vv \tau \\ \xi' &\vv \xi \end{matrix}\right)
  =\A^{\xi'} \wt Q \left(\begin{matrix} \vv \tau \\  \vv \xi \end{matrix}\right),
\end{equation}
\begin{equation}
  \label{Qb}
    \wt Q \left(\begin{matrix}\vv \tau & 1 \\ \vv \xi & \xi' \end{matrix}\right)
  - q\,\B \wt Q \left(\begin{matrix}\vv \tau & 0 \\ \vv \xi & \xi' \end{matrix}\right)
  =\B^{1-\xi'} \wt Q \left(\begin{matrix} \vv \tau \\  \vv \xi \end{matrix}\right),
\end{equation}
\begin{equation}
  \label{Q:bulk}
  \wt Q \left(\begin{matrix}\vv \tau_1 & 1 &0 &\vv \tau_2 \\ \vv \xi_1&\xi'&\xi'' & \vv \xi_2 \end{matrix}\right)
  -q\,\wt Q \left(\begin{matrix}\vv \tau_1 & 0 &1 &\vv \tau_2 \\ \vv \xi_1&\xi'&\xi'' & \vv \xi_2 \end{matrix}\right) = \wt Q \left(\begin{matrix}\vv \tau_1 & 1-\xi''  &\vv \tau_2 \\ \vv \xi_1&\xi'&  \vv\xi_2 \end{matrix}\right),
\end{equation}
where $\xi',\xi''\in\{0,1\}$  and $\vv \xi,\vv \tau, (\vv \tau_1,\vv\tau_2)$, $(\vv \xi_1,\vv  \xi_2)\in\Omega_L$, $L=0,1,\dots$.
\end{theorem}
 In the above formulas, we write
$\begin{pmatrix}
  \vv \tau_1 & \vv \tau_2& \vv \tau_3\\
  \vv \xi_1 & \vv \xi_2& \vv \xi_3
\end{pmatrix}$
to indicate the concatenations $(\vv \tau_1,\vv \tau_2,\vv \tau_3)$ and $(\vv \xi_1,\vv \xi_2,\vv \xi_3)$ of the top and bottom rows of the  sequences of $\{0,1\}$  that indicate the locations of the particles in each row.
\arxiv
{In particular, solving the system of equations \eqref{Qa} \eqref{Qb}  for $L=0$, we see that   recursions (\ref{Qa}-\ref{Q:bulk})  can be initialized at $L=1$ with
$$\wt Q\begin{pmatrix} 0 \\
        \xi
      \end{pmatrix}= \frac{ \A^{\xi }   +q\,\A\B^{1-\xi }}{1-\A\B q^2},
   \quad \wt Q\begin{pmatrix} 1 \\
        \xi
      \end{pmatrix}  = \frac{
   \B^{1-\xi}+q\, \A^{\xi }\B}{1-\A\B q^2}.$$
   (Note the singularity when $\A\B q^2=1$.)
   The values of $Q$ computed from \eqref{Q-def}  take the following  matching form:
       $$Q\begin{pmatrix} 0 \\
        0
      \end{pmatrix}= 1+q\,\A\B,
   \quad
Q\begin{pmatrix} 0 \\
        1
      \end{pmatrix}=\A(1+q),
   \quad
    Q\begin{pmatrix} 1 \\
        0
      \end{pmatrix}  = \B(1+q), \quad
    Q\begin{pmatrix} 1 \\
        1
      \end{pmatrix}  = 1+q\,\A\B .
   $$
(This is $L=1$ case of \eqref{wt Q}. The dependence on $q$ is due to Liggett's condition \eqref{Liggets-greek})

}

 The proof of Theorem \ref{thm2} is given in Section \ref{Sec:ProofThm2}.
We now deduce Theorem \ref{thm1} from Theorem \ref{thm2}.
\begin{proof}[Proof of Theorem \ref{thm1}]
   Consider first the nonsingular case $\A,\B \in [0,\infty) \setminus \{q^{-j}: j=0,1,\dots\}$.
From \eqref{wtW2PTL}, it is clear that formula \eqref{TL2mu} will follow if  the basic weight equations (\ref{pini}-\ref{pbulk}) hold for
 \begin{equation}
   \label{Phi-sol}
    \Phi(\vv \tau) =  \sum_{\xi\in\Omega_L} \wt Q\begin{pmatrix}
   \vv \tau \\ \vv \xi
 \end{pmatrix}.
 \end{equation}
Equation \eqref{pini} holds by definition.
       Equations \eqref{pa}, \eqref{pb}, and \eqref{pbulk} are immediate consequences of
   \eqref{Qa}, \eqref{Qb} and   \eqref{Q:bulk} respectively.
   \arxiv
   {
   Indeed,
  \begin{multline*}
    \Phi(0,\vv \tau) - q\,\A\,\Phi(0,\vv \tau)
    = \sum_{\xi'\in\{0,1\}}\sum_{\vv \xi \in \Omega_L}
    \left( \wt Q \begin{pmatrix} 0& \vv \tau  \\
     \xi'&  \vv \xi  \end{pmatrix}
     - q\,\A\ \wt Q \begin{pmatrix} 1& \vv \tau  \\
     \xi'&  \vv \xi  \end{pmatrix}\right)
     \\ =
      \sum_{\xi'\in\{0,1\}} \A^{\xi'}\sum_{\vv \xi \in \Omega_L} \wt Q
   \begin{pmatrix}  \vv \tau  \\
       \vv \xi  \end{pmatrix}=(1+\A) \sum_{\vv \xi \in \Omega_L}\wt Q
   \begin{pmatrix}  \vv \tau  \\
       \vv \xi  \end{pmatrix}=(1+\A)\Phi(\vv\tau).
  \end{multline*}
  Similarly,
    \begin{multline*}
    \Phi(\vv \tau,1) - q\,\B\,\Phi(\vv \tau,0)
    =  \sum_{\vv \xi \in \Omega_L,\xi'\in\{0,1\}}
    \left(\wt Q \begin{pmatrix}
       \vv \tau & 1 \\
       \vv \xi & \xi'
    \end{pmatrix} - q\,\B\, \wt Q  \begin{pmatrix}
       \vv \tau & 0 \\
       \vv \xi & \xi'
    \end{pmatrix}\right)
    =   \sum_{\vv \xi \in \Omega_L}
    \wt Q
   \begin{pmatrix}  \vv \tau  \\
       \vv \xi  \end{pmatrix} \sum_{\xi'\in\{0,1\}} \B^{1-\xi'} =(1+\B)\Phi(\vv\tau)
    \end{multline*}
    Finally,
    \begin{multline*}
       \Phi(\vv \tau_1,1,0,\vv\tau_2) - q\, \Phi(\vv \tau_1,1,0,\vv\tau_2)
       \\
       =  \sum_{\xi',\xi''\in\{0,1\}}  \sum_{\vv \xi_1,\vv\xi_2}
       \left(
        \wt Q \begin{pmatrix}\vv \tau_1 & 1 &0 &\vv \tau_2 \\ \vv \xi_1&\xi'&\xi'' & \xi_2 \end{pmatrix}
  -q\,\wt Q \begin{pmatrix}\vv \tau_1 & 0 &1 &\vv \tau_2 \\ \vv \xi_1&\xi'&\xi'' & \xi_2 \end{pmatrix}
       \right)
       \\ =
            \sum_{\vv \xi_1,\vv\xi_2} \sum_{\xi',\xi''\in\{0,1\}} \wt Q \begin{pmatrix}\vv \tau_1 & 1-\xi''  &\vv \tau_2 \\ \vv \xi_1&\xi'&  \xi_2 \end{pmatrix}
        =   \sum_{\vv \xi_1,\vv\xi_2,\xi '}   \left(\wt Q \begin{pmatrix}\vv \tau_1 & 0&\vv \tau_2 \\ \vv \xi_1&\xi'&  \xi_2 \end{pmatrix}+
        \wt Q \begin{pmatrix}\vv \tau_1 & 1&\vv \tau_2 \\ \vv \xi_1&\xi'&  \xi_2 \end{pmatrix}\right)
       \\ =\Phi(\vv\tau_1, 0,\vv \tau_2)+\Phi(\vv\tau_1, 1,\vv \tau_2).
    \end{multline*}
   }
This proves that \eqref{TL2mu} holds in the nonsingular case.

  Since the common multiplicative factor $(\A\B;q)_2/(\A\B;q)_{L+2}$ appears in all expressions on the right-hand side of \eqref{wtW2PTL}, including the normalizing constant $\wt Z$, after cancellation we see that \eqref{TL2mu} holds for all $\A,\B \in [0,\infty) \setminus \{q^{-j}: j=2,3,\dots\}$ with $P_{\rm TL}$ given by \eqref{PTL1}. However, by Lemma \ref{lem1} expression \eqref{Q-def} is a polynomial in variables $\A,\B$ so $Q$ is well defined and nonnegative for all $\A,\B\geq 0$, and \eqref{PTL1} gives positive normalization constant $Z_L(\A,\B,q)$. The invariant measure $\mu$ is a continuous function of the parameters $\A,\B>-1$ so, by continuity, formula \eqref{TL2mu} is valid for all $\A,\B\geq 0$.
\end{proof}

\begin{remark}
Representation \cite{Bryc-Zatitskii-2024} led to simple proofs of large deviations and KPZ-fixed-point asymptotics for TASEP and allowed us to analyze fluctuations on the coexistence line in Ref. \cite{Bryc-Najnudel-Wang-2024}. It would be interesting to investigate whether  Theorem \ref{thm1} could be used in a similar way    for ASEP.

 Of course, most of these results are known for the general ASEP. The large deviations for ASEP in the form that distinguishes between the fan region $\A\B<1$ and the shock region $\A\B> 1$ are known from \cite{derrida2003-Exact-LDP}. The KPZ-fixed-point limit for ASEP with no restrictions on the parameters appeared in \cite{wang2024asymmetric}. The asymptotic regime for the convergence of the height function of ASEP in the fan region $\A\B<1$ to the stationary measure of the open KPZ equation is described in \cite{corwin-knizel2021stationary}.  In order to use Theorem \ref{thm1} to recover some of these results, one would need to investigate the asymptotic properties of polynomials $ w_{\vv \sigma}(z)$ for random compositions $\vv \sigma$ under the uniform law on $\Omega_L\times\Omega_L$. Unfortunately,  the literature seems to concentrate on random compositions $\vv \sigma$ with a uniform law on the set of all compositions, which is not the law we have in Theorem \ref{thm1}.
 (A nice probabilistic description of uniformly distributed compositions  appears in  \cite[Section 3]{hitczenko2004multiplicity}.)

\end{remark}

\begin{remark}
 We note that the support of $P_{\rm TL}$  is $\Omega_L\times\Omega_L$ if $\A,\B>0$. If one of the parameters is 0, then the support of $P_{\rm TL}$  is  a subset of $\Omega_L\times\Omega_L$. If $\A=0$ then the measure $P_{\rm TL}$ is supported on $(\vv \tau,\vv\sigma)$ which correspond to nonnegative paths $\vv \gamma$ and if $\B=0$ then it is supported on $(\vv\tau,\vv\xi)$  such that the corresponding path $\vv\gamma=\gamma(\vv \tau,\vv \xi)$ attains its minimum at the end, $\gamma_L=m(\vv \gamma)$.
In particular, for $\A=\B=0$ measure $P_{\rm TL}$ is supported on the subset $\calC$ of $\Omega_L\times\Omega_L$ such that $\vv \gamma(\vv \tau,\vv\ \sigma)$ is a Motzkin path, i.e., $\gamma_j\geq0$ and $\gamma_L=0$.  Since $[1]_q=1$, in this case we get
\begin{equation}
  \label{P4AB=0}
  P_{\rm TL}\left(\begin{matrix}\vv \tau \\ \vv \xi \end{matrix}\right)=\begin{cases}
   \frac{1}{Z}  \displaystyle\prod_{j=1}^r ([j+1]_q)^{\sigma_j} & \mbox{ if }  \gamma(\vv \tau,\vv \xi) \mbox{ is a Motzkin path}, \\
   0 & \mbox{ otherwise}.
  \end{cases}.
\end{equation}
When $q=0$ this becomes a uniform measure on essentially the same {\em configuration space} $\calC\subset\Omega_L\times\Omega_L$ that appeared in \cite{duchi2005combinatorial}.
For general $q$, this formula gives explicit separation of two layers in \cite[(2.10)  and (2.12)]{nestoridi2023approximating} for their parameters $u=v=-q$, which is equivalent to setting our parameters $\A=\B=0$.
\end{remark}
\subsection{Comparison with another two-layer  model}\label{sec:compare}
Consider the set $\calC$ of $\Omega_L\times\Omega_L$ defined by the requirement that $\vv \gamma(\vv \tau,\vv\xi)$ is a Motzkin path. %
In a very interesting paper \cite{duchi2005combinatorial} defined a Markov evolution on   $\calC$ which coincides with the evolution of TASEP on the top layer. They also determined the invariant measure on the two layers, which becomes a uniform measure on $\calC$ when $\A=\B=0$.

The original expression for the stationary measure in \cite{duchi2005combinatorial} is quite involved. We therefore translate  the description of this measure given in \cite[Section 2.7]{Bertini-2007} into our notation.
According to \cite[Section 2.7]{Bertini-2007}, we label the sites $j\in\{1,\dots,L\}$ as W, if
$\gamma_{j-1}=\gamma_j=0$ and $\xi_j=1$ (then necessarily $\tau_j=1$) and label them B if
$\gamma_{j-1}=0$, $\xi_j=0$  and there are no W-labeled sites to the left of $j$. The remaining sites are not labeled. We denote by $N(W)$ the number of W-labeled site and by $N(B)$ the number of B-labeled sites. The two-layer measure on $\calC$ is then given by normalizing the weight function
\begin{equation}
  \label{P-Duchi}
\calQ\begin{pmatrix}
  \vv \tau \\ \vv \xi
\end{pmatrix} = (1+\A)^{N(B)}(1+\B)^{N(W)}.
\end{equation}
\begin{figure}[htb]
\begin{tabular}{lr}
  \begin{tikzpicture}[scale=.5]
  \draw [fill=black, ultra thick] (.5,1.8) circle [radius=0.2];
  \draw [ultra thick] (1.5,1.8) circle [radius=0.2];
\draw [fill=black, ultra thick] (2.5,1.8) circle [radius=0.2];
\draw [fill=black, ultra thick] (3.5,1.8) circle [radius=0.2];
\draw [ultra thick] (4.5,1.8) circle [radius=0.2];
  \draw [ultra thick] (5.5,1.8) circle [radius=0.2];
   \draw [fill=black, ultra thick] (6.5,1.8) circle [radius=0.2];
    \draw [ultra thick] (7.5,1.8) circle [radius=0.2];
      \draw [ultra thick] (8.5,1.8) circle [radius=0.2];
   \draw [ultra thick] (9.5,1.8) circle [radius=0.2];
\draw [ ultra thick] (.5,1) circle [radius=0.2];
  \draw [fill=black,ultra thick] (1.5,1) circle [radius=0.2];
\draw [fill=black, ultra thick] (2.5,1) circle [radius=0.2];
\draw [fill=black, ultra thick] (3.5,1) circle [radius=0.2];
\draw [ultra thick] (4.5,1) circle [radius=0.2];
  \draw [ultra thick] (5.5,1) circle [radius=0.2];
   \draw [fill=black, ultra thick] (6.5,1) circle [radius=0.2];
    \draw [ultra thick] (7.5,1) circle [radius=0.2];
      \draw [fill=black, ultra thick] (8.5,1) circle [radius=0.2];
   \draw [fill=black,ultra thick] (9.5,1) circle [radius=0.2];
   \draw [-, ultra thick] (0,.6) to (0,2.2);
\draw [-, thick] (1,.6) to (1,2.2);
\draw [-, thick] (2,.6) to (2,2.2);
\draw [-, thick] (3,.6) to (3,2.2);
\draw [-, thick] (4,.6) to (4,2.2);
\draw [-, thick] (5,.6) to (5,2.2);
\draw [-, thick] (6,.6) to (6,2.2);
\draw [-, thick] (7,.6) to (7,2.2);
\draw [-, thick] (8,.6) to (8,2.2);
\draw [-, thick] (9,.6) to (9,2.2);
\draw [-, ultra thick] (10,.6) to (10,2.2);
   \draw [-, thick] (0,1.4) to (10,1.4);
    \draw [-, thick] (0,.6) to (10,.6);
   \draw [-, ultra thick] (0,.6) to (10,.6);
      \draw [-, ultra thick] (0,2.2) to (10,2.2);
     \node  at (.5,0.2) {B};
      \node  at (1.5,0.2) {W};
       \node  at (9.5,0.2) {W};
\end{tikzpicture}
&
  \begin{tikzpicture}[scale=.5]
  \draw [fill=black, ultra thick] (.5,1.8) circle [radius=0.2];
  \draw [ultra thick] (1.5,1.8) circle [radius=0.2];
\draw [fill=black, ultra thick] (2.5,1.8) circle [radius=0.2];
\draw [fill=black, ultra thick] (3.5,1.8) circle [radius=0.2];
\draw [ultra thick] (4.5,1.8) circle [radius=0.2];
  \draw [ultra thick] (5.5,1.8) circle [radius=0.2];
   \draw [fill=black, ultra thick] (6.5,1.8) circle [radius=0.2];
    \draw [ultra thick] (7.5,1.8) circle [radius=0.2];
      \draw [ultra thick] (8.5,1.8) circle [radius=0.2];
   \draw [ultra thick] (9.5,1.8) circle [radius=0.2];
\draw [fill=blue, ultra thick] (.5,1) circle [radius=0.2];
  \draw [ultra thick] (1.5,1) circle [radius=0.2];
\draw [ ultra thick] (2.5,1) circle [radius=0.2];
\draw [ ultra thick] (3.5,1) circle [radius=0.2];
\draw [fill=blue,ultra thick] (4.5,1) circle [radius=0.2];
  \draw [fill=blue,ultra thick] (5.5,1) circle [radius=0.2];
   \draw [ ultra thick] (6.5,1) circle [radius=0.2];
    \draw [fill=blue,ultra thick] (7.5,1) circle [radius=0.2];
      \draw [ ultra thick] (8.5,1) circle [radius=0.2];
   \draw [ultra thick] (9.5,1) circle [radius=0.2];
   \draw [-, ultra thick] (0,.6) to (0,2.2);
\draw [-, thick] (1,.6) to (1,2.2);
\draw [-, thick] (2,.6) to (2,2.2);
\draw [-, thick] (3,.6) to (3,2.2);
\draw [-, thick] (4,.6) to (4,2.2);
\draw [-, thick] (5,.6) to (5,2.2);
\draw [-, thick] (6,.6) to (6,2.2);
\draw [-, thick] (7,.6) to (7,2.2);
\draw [-, thick] (8,.6) to (8,2.2);
\draw [-, thick] (9,.6) to (9,2.2);
\draw [-, ultra thick] (10,.6) to (10,2.2);
   \draw [-,  thick] (0,1.4) to (10,1.4);
    \draw [-, ultra thick] (0,.6) to (10,.6);
      \draw [-, ultra thick] (0,2.2) to (10,2.2);
      \draw [-,gray,dotted] (-.1,3) to (10.1,3);
       \draw [-,blue] (0,3) to (2,3);
        \draw [-,blue] (2,3) to (3,4);
        \draw [-,blue] (3,4) to (4,5);
          \draw [-,blue] (4,5) to (6,3);
           \draw [-,blue] (6,3) to (7,4);
               \draw [-,blue] (7,4) to (8,3);
                \draw [-,blue] (8,3) to (10,3);
\end{tikzpicture}
\end{tabular}
\caption{\textbf{Left:} An example of a two layer configuration as in \cite{duchi2005combinatorial} for $L=10$, with labeling of the bottom row needed for \eqref{P-Duchi}.  Formula \eqref{P-Duchi} assigns weight $\calQ\begin{pmatrix}
  \vv \tau \\ \vv \xi
\end{pmatrix}  =\ (1+\A)(1+\B)^2$ to this configuration. \textbf{Right:} The equivalent two-layer configuration in our notation with the locations of particles
$\vv \tau=(1,0,1,1,0,0,1,0,0,0)$ and $\vv \xi=(1,0,0,0,1,1,0,1,0,0)$. The random walk path $\vv\gamma=(0,0,0,1,2,1,0,1,0,0,0)$ is drawn as a continuous interpolation of the function $j\mapsto \gamma_j$ and gives composition $\vv \sigma(\vv \gamma)=(7,3,1)$.
With $q=0$, formula \eqref{Q-def} assigns the
weight $Q\begin{pmatrix}
  \vv \tau \\ \vv \xi
\end{pmatrix}=1$ that does not depends on $\A,\B$ to this configuration.
\label{Fig:2L}}
\end{figure}
In general, their invariant measure is different than ours, as it is supported only on $\calC$ for all $\A,\B$, while the support of our two-layer measure $P_{\rm TL}$ for $q=0$ and $\A,\B>0$ is the entire $\Omega_L\times\Omega_L$. However,  when $\A=\B=0$ and $q=0$ both \eqref{P-Duchi} and \eqref{P4AB=0} define the same uniform measure on $\calC$.  (To compare with their setup, our $\vv\tau\in\Omega_L$ describes the positions of particles in their top layer, and our $\vv \xi\in\Omega_L$ describes the positions of holes in the bottom layer of their model, as indicated on the left and right sides of Figure \ref{Fig:2L}.)

For $q=0$, formulas \eqref{Q-def} and \eqref{P-Duchi}  represent the same probability  $\mu(\tau)$  as the (normalized) sums of polynomials in $\A,\B$. Representation of $\mu$ based on formula \eqref{Q-def} uses monomials in $\A,\B$.  Representation based on \eqref{P-Duchi} uses monomials in $1+\A$, $1+\B$.

 Formula
\eqref{P4AB=0}   extends \eqref{P-Duchi} to $q>0$ in the case $\A=\B=0$.  We remark that \cite{Corteel-Williams-2007-Markov,Corteel-Williams-2007-AAM} generalize \cite{duchi2005combinatorial} to $q>0$ by introducing a larger configuration space that consists  of  the {\em
staircase tableaux} that they introduce.

\section{Proofs}
Our proofs exploit the well-known properties of the Jackson derivative $\Dq$.  This operator satisfies the basic $q$-commutation identity
\begin{equation}\label{D-q-comm}
  \Dq z- q\,z \Dq=1, \quad z\in\RR\setminus\{0\},
\end{equation}
which is a special case of the $q$-product rule
\begin{equation}
  \label{q-prod}
  \Dq [f(z)g(z)]=g(z)\Dq[f(z)]+f(q z) \Dq[g(z)].
\end{equation}
It is also known (and easy to check) that
\begin{equation}
  \label{Dq-prop}
  \Dq[z^n]=[n]_q z^{n-1},\quad
    \Dq\frac{1}{(z;q)_n}=\frac{[n]_q}{(z;q)_{n+1}}, \quad n=0,1,\dots.
\end{equation}
A good reference for the definition and most of the formulas above is \cite[(1.5) and (1.10)-(1.12)]{Kac-2002}. Unfortunately, this reference does not consider the $q$-Pochhammer symbol, so  we verify the last formula in \eqref{Dq-prop} for completeness.
From $(a;q)_{n+1}=(a;q)_n(1-a q^n)$ and
$(q a;q)_n=(a;q)_{n+1}/(1-a )$ we get
$$
  \Dq\frac{1}{(z;q)_n}= \frac{\frac{1}{(z;q)_n} - \frac{1}{(qz;q)_n}}{(1-q)z}
=\frac{1}{(z;q)_n(1-q)z}\left(1-\frac{(1-z)}{1-z q^n}\right)
=\frac{1}{(z;q)_{n+1}}\frac{z(1-q^n)}{(1-q)z}
= \frac{[n]_q}{(z;q)_{n+1}}.
$$

\begin{lemma}\label{lem1}Let $\vv\sigma$ be a composition of $L+1$ into $r+1$ parts for some $L\in\ZZ_{\geq 0}$.
For $z\ne 0, 1,1/q,\dots$ and $0\leq q<1$ let $w_{\vv \sigma}$ be defined by \eqref{w2Dq}.
Then $w_{\vv \sigma}(z)$ is a polynomial in variable $z$ with nonnegative coefficients and of degree at most $L-r$.
In particular, $ w_{\vv \sigma}(z)$ extends by continuity to all real $z$.
Furthermore,   for $|z|<1$, this polynomial is given by \eqref{w-def}.
\end{lemma}

\arxiv
{The degree of $ w_{\vv \sigma}(z)$ is of course $0$ when $q=0$.
The proof shows that the degree of $ w_{\vv \sigma}(z)$ is actually $L-r$ if $q>0$.
For example, with $\sigma=(1,1,\dots,1)$, that is, $L-r=0$ we have
\begin{multline*}
 w_{\vv \sigma}(z)=(z;q)_{L+2} \Dq^{L+1} \frac{1}{(z;q)_1}=
(z;q)_{L+2} \Dq^{L} \frac{[1]_q}{(z;q)_2}=(z;q)_{L+2} \Dq^{L-1} \frac{[1]_q[2]_q}{(z;q)_3}\dots =(z;q)_{L+2} \frac{[L+1]_q!}{(z;q)_{L+2}}=[L+1]_q!
\end{multline*}
On the other hand, with $\sigma=(L+1)$ i.e. $L-r=L$ polynomial
$$ w_{\vv \sigma}(z)=(z;q)_{L+2} (\Dq z)^{L} \frac{1+q}{(z;q)_2} =
(z;q)_{L+2}(1+q z \Dq)^{L}\frac{1+q}{(z;q)_2}=(z;q)_{L+2}\frac{1+q}{(z;q)_2}+ \mbox{lower order terms} $$
has degree $L$ if $q>0$. (Here we used \eqref{D-q-comm}.)
}

\begin{proof}
To prove that \eqref{w2Dq} defines a polynomial in variable $z$ with nonnegative coefficients and of degree at most $L-r$,
we prove by mathematical induction that %
\begin{equation}
  \label{ind-w}
  \frac{ w_{\vv \sigma}(z)}{(z;q)_{L+2}}
\mbox{ is in the nonnegative span of } \left\{\frac{z^j}{(z;q)_{L+2}}: j =0,\dots,L-r\right\}.
\end{equation} The induction is on the number of parts $r+1$ of composition $\vv \sigma$.

If $r=0$, then $L+1=\sigma_0$ is a composition with 1 part and the fact that $(\Dq z)^{\sigma_0-1}\Dq \tfrac{1}{1-z}$ is in the nonnegative span of $\{\frac{z^j}{(z;q)_{L+2}}: j =0,\dots,L\}$ follows by the same argument that we use in the induction step. (Alternatively, one  can start the induction with $r=-1$, where the composition of $L+1=0$ has 0 parts and $w_\emptyset(z)=1$.)

 Suppose that \eqref{ind-w} holds for some $r\geq 0$ and all integers $L\geq 0$ and all compositions $\vv \sigma$, of $L+1$ into $r+1$ parts.

Let $r'=r+1$, and let $\vv \sigma'=(\sigma'_0,\dots,\sigma'_{r'})$  be a sequence of positive integers which forms a composition of $L'+1:=\sigma_0'+\dots+\sigma'_{r'}$ into $r'+1=r+2$ parts. Then $\vv \sigma'=(\sigma_0',\vv \sigma)$ where $\vv\sigma:=(\sigma_1',\dots,\sigma_{r+1}')$ is a composition of $L+1$ into $r+1$ parts with $L:=L'-\sigma_0'$.
Formula \eqref{w2Dq} shows that
\begin{equation}
  \label{w'2w}
  \frac{w_{\vv \sigma'}(z)}{(z;q)_{L'+2}}= (\Dq z)^{\sigma_0'-1}\Dq \frac{ w_{\vv \sigma}(z)}{(z;q)_{L+2}},
\end{equation}
and by induction assumption, $ w_{\vv \sigma}(z)$ is a polynomial in $z$ of degree $L-r$ with nonnegative coefficients.
Therefore, it is enough to verify the effect of the action of the operator $(\Dq z)^{\sigma_0'-1}\Dq$ on the expressions of the form
${z^j}/{(z;q)_{L+2}}$ for $j=0,1,\dots,L-r$.

Using $q$-product formula \eqref{q-prod} and the second formula in \eqref{Dq-prop}, we see that
$$
\Dq \frac{z^j}{(z;q)_{L+2}}=
\frac{[j]_q z^{j-1}}{(z;q)_{L+2}}+ q^j z^j \frac{[L+2]_q}{(z;q)_{L+3}}
   =\frac{1}{(z;q)_{L+3}}\left(q^j [L+2-j]_qz^j +
   [j]_qz^{j-1}\right),
   $$
 an expression which is in the nonnegative span of $\frac{z^j}{(z;q)_{L+3}}$, $j=0,\dots,L-r$.
By the same argument, for $j=0,\dots,L-r$, we have
$$\Dq z \frac{z^j}{(z;q)_{L+3}}
=\frac{[j+1]_q z^{j}}{(z;q)_{L+3}}+ q^{j+1} z^{j+1} \frac{[L+3]_q}{(z;q)_{L+4}}
=\frac{1}{(z;q)_{L+4}}\left(q^{j+1}  [L+2-j]_q z^{j+1}+[j+1]_q z^j
\right)
$$
is in the nonnegative span of functions $\frac{z^j}{(z;q)_{L+3}}$, $j=0,\dots,L-r+1$. (Moreover, if $q\ne 0$ then the coefficient at $z^{j+1}$  is positive, so the resulting polynomial is of degree $L-r+1$.)

Since formula \eqref{w'2w} has  $\sigma_0'-1$ iterations of operator $\Dq z$, and each of them raises the range of the  powers of $z$ in the numerator and the length of the $q$-Pochhammer symbol in the denominator by 1, we see that if $j\in\{0,\dots,L-r\}$
then $(\Dq z)^{\sigma_0'-1}\Dq \frac{z^j}{(z;q)_{L+2}}$ is in the nonnegative span of
 $$\left\{\frac{z^j}{(z;q)_{L+2+\sigma_0'}}: j=0,\dots,L-r+\sigma_0'-1\right\}
 =\left\{\frac{z^j}{(z;q)_{L'+2}}: j=0,\dots,L'-r' \right\}.$$
This completes the induction step and ends the proof. (We also proved that if $q\ne 0$ then the degree of polynomial $ w_{\vv \sigma}(z)$ is $L-r$.)

We now prove that the right hand sides of \eqref{w2Dq} and \eqref{w-def} coincide for $0<|z|<1$. (Condition $|z|<1$   ensures convergence of the series \eqref{w-def}. Condition $z\ne0$ is required in definition \eqref{Dq-def} of $\Dq$.) Peeling off one rightmost operator at a time we get
\begin{multline*}
\prod_{j=0}^r \left( (\Dq z)^{\sigma_j-1}\Dq\right) [\tfrac{z^{r+1}}{1-z}]=
\sum_{n=0}^\infty\prod_{j=0}^r \left( (\Dq z)^{\sigma_j-1}\Dq\right) z^{n+r+1}
\\
=
  \sum_{n=0}^\infty \prod_{j=0}^{r-1} \left( (\Dq z)^{\sigma_j-1}\Dq\right) (\Dq z)^{\sigma_r-1}\Dq z^{n+r+1}
  \\=
 \sum_{n=0}^\infty \prod_{j=0}^{r-1} \left( (\Dq z)^{\sigma_j-1}\Dq\right) (\Dq z)^{\sigma_r-2} \Dq z [n+r+1]_q z^{n+r}
 \\=
  (z;q)_{L+1} \sum_{n=0}^\infty \prod_{j=0}^{r-1} \left( (\Dq z)^{\sigma_j-1}\Dq\right) (\Dq z)^{\sigma_r-2}  z^{n+r} [n+r+1]_q^2
  \\=  \sum_{n=0}^\infty \prod_{j=0}^{r-1} \left( (\Dq z)^{\sigma_j-2}\Dq\right) (\Dq z)^{\sigma_{r-1}-1}\Dq   z^{n+r} [n+r+1]_q^{\sigma_r}
 \\ =
  \sum_{n=0}^\infty \prod_{j=0}^{r-1} \left( (\Dq z)^{\sigma_j-2}\Dq\right) (\Dq z)^{\sigma_{r-1}-1}z^{n+r-1}[n+r]_q  [n+r+1]_q^{\sigma_r}
  \\ =
  \sum_{n=0}^\infty \prod_{j=0}^{r-1} \left( (\Dq z)^{\sigma_j-2}\Dq\right)  z^{n+r-1}[n+r]_q^{\sigma_{r-1}}  [n+r+1]_q^{\sigma_r}
  \\ =\dots=
    \sum_{n=0}^\infty  (\Dq z)^{\sigma_0-1}\Dq   z^{n+1}  \prod_{j=1}^{r} [n+j+1]_q^{\sigma_j} = \sum_{n=0}^\infty     z^{n}  \prod_{j=0}^{r} [n+j+1]_q^{\sigma_j}.
\end{multline*}
Noting that
$$
\frac{z^{r+1}}{1-z}=\frac{1}{1-z}+\frac{z^{r+1}-1}{1-z}
$$
differs from $(1-z)^{-1}$ by a polynomial of degree $r$,
and that $\prod_{j=0}^r \left( (\Dq z)^{\sigma_j-1}\Dq\right)$ reduces the degree of a polynomial by $r+1$ we see that \eqref{w2Dq} and \eqref{w-def} give the same expression.
\end{proof}

\subsection{Proof of Theorem \ref{thm2}}\label{Sec:ProofThm2}
Since  \eqref{Qa} and \eqref{Q:bulk}  are identities between polynomials in variables $\A,\B\geq 0$, it suffices to prove them for $\A,\B>0$ only. We therefore assume that $\A>0$ and $\B>0$  throughout the proof.
It is convenient to rewrite \eqref{wt Q} as
\begin{equation*}
\wt Q \left(\begin{matrix}\vv \tau \\ \vv \xi \end{matrix}\right)=(\A\B;q)_2  g_{\A,\B}\begin{pmatrix}
\vv \tau \\ \vv \xi
\end{pmatrix} \wt w_{\vv \sigma}(\A\B),
\end{equation*}
where $\vv \sigma=\sigma(\vv \gamma)$, with $\vv \gamma=\gamma(\vv \tau,\vv \xi)$, $\wt w_{\vv \sigma}(z)= w_{\vv \sigma}(z)/(z;q)_{L+2}$, and
\begin{equation}
  \label{q=0}
   g_{\A,\B}\begin{pmatrix}
\vv \tau \\ \vv \xi
\end{pmatrix}=\frac{\B^{\gamma_L}}{(\A\B)^{\min(\vv \gamma)}}.
\end{equation}
This allows us to separate the contribution of $\wt w_{\vv \sigma}(z)$ to the identities, and to use \eqref{w2Dq}  to complete the proofs.
 For ease of reference we note that \eqref{w2Dq} is the same as
 \begin{equation}
  \label{w2Dq'}
 \wt w_{\vv \sigma}(z)=\prod_{j=0}^r \left( (\Dq z)^{\sigma_j-1}\Dq\right) [\tfrac{1}{1-z}].
\end{equation}

To keep track of modifications that we need to apply to the top and bottom rows in the arguments of $\vv\gamma$,  we will  use the two-row  notation
$\gamma\begin{pmatrix}
  \vv \tau \\  \vv \xi
\end{pmatrix}$
instead of $\gamma(\vv \tau,\vv \xi)$. As previously, we use boldface $\vv \gamma$ for the sequence and standard font for the function $\gamma(\cdot,\cdot)$ and for the values $\gamma_0,\dots,\gamma_L$ of the sequence.

The derivations of the boundary identities \eqref{Qa},
\eqref{Qb} and the bulk identity \eqref{Q:bulk} are similar but the
details differ, and there are several exceptional cases that need
to be considered separately.

\subsection*{Boundary identities }
We fix $\vv \tau,\vv\xi\in\Omega_L$, $\vv \gamma=\gamma(\vv \tau, \vv \xi)$ and $\vv \sigma=\sigma(\vv \gamma)$. The details of the proof will depend on the values of
$m=\min(\vv \gamma)$, and $M=\max(\vv\gamma)$. We will go in detail over the generic case, where $m$, $M$ and the endpoints of $\vv\gamma$ do not coincide. We omit the arguments for the exceptional boundary cases $m=0$,  $M=0$, $m=\gamma_L$ or $M=\gamma_L$ which require (often minor) changes.
(The omitted details are in the expanded version of the paper on arxiv.)

\subsubsection{Left boundary}
We begin with the proof of \eqref{Qa}.
The first observation is that
\begin{equation*}
  \label{g0L}
  g_{\A,\B}\begin{pmatrix}
  0 & \vv \tau \\ \xi' & \vv \xi
\end{pmatrix}=\A^{\xi'}g_{\A,\B}\begin{pmatrix}
   \vv \tau \\  \vv \xi
\end{pmatrix}
\end{equation*}
and
\begin{equation}
  \label{g1L}
  g_{\A,\B}\begin{pmatrix}
  1 & \vv \tau \\ \xi' & \vv \xi
\end{pmatrix}=\begin{cases}
\B\, g_{\A,\B}\begin{pmatrix}
   \vv \tau \\  \vv \xi
\end{pmatrix} & \mbox{ if }  m=0 \mbox{ and } \xi'=0, \\
 \A^{\xi'-1}g_{\A,\B}\begin{pmatrix}
   \vv \tau \\  \vv \xi
\end{pmatrix} & \mbox{ otherwise} .   \\
\end{cases}
\end{equation}
(Recall that $\A>0$ throughout this proof.)
This follows   by going over the cases listed in Table \ref{TabLeft} and inspecting how  formula \eqref{q=0} changes in each case.

\begin{table}[htb]
  \begin{tabular}{||c||c|c||} \hline\hline
  $\xi'$ & $\vv \gamma\begin{pmatrix}
     0 &\vv \tau
    \\  \xi' & \vv \xi
  \end{pmatrix}$ & $\vv \gamma\begin{pmatrix}
     1 &\vv \tau
    \\  \xi' & \vv \xi
  \end{pmatrix}$ \\
  \hline\hline
  $0$ &  \begin{tikzpicture}[scale=1]
        \draw [-,blue,thick] (0.5,0.5) to (1,.5);
        \draw [fill=black] (1.0,.5) circle [radius=0.03];
             \draw [fill=black] (0.5,0.5) circle [radius=0.03];
         \draw [-,gray,dotted,thick] (1,.5) to (1.25,.75);
        \draw [-,gray,dotted,thick] (1.25,.75) to (1.5,.5);
                 \draw [fill=white] (1.25,-0.2) circle [radius=0.0];
                  \draw [fill=white] (1.25,1) circle [radius=0.0];
         \draw [-,gray,dotted,thick] (1.5,.5) to (1.75,-.5);
\draw [-,gray,dotted,thick] (1.75,-.5) to (1.9,1.1);
\draw [-,gray,dotted,thick] (1.9,1.1) to (2,.6);
  \end{tikzpicture}
  &

  \begin{tikzpicture}[scale=1]
        \draw [-,blue,thick] (0.5,0) to (1,.5);
        \draw [fill=black] (1.0,.5) circle [radius=0.03];
             \draw [fill=black] (0.5,0) circle [radius=0.03];
         \draw [-,gray,dotted,thick] (1,.5) to (1.25,.75);
        \draw [-,gray,dotted,thick] (1.25,.75) to (1.5,.5);
                  \draw [fill=white] (1.25,-0.2) circle [radius=0.0];
                  \draw [fill=white] (1.25,1) circle [radius=0.0];
         \draw [-,gray,dotted,thick] (1.5,.5) to (1.75,-.5);
\draw [-,gray,dotted,thick] (1.75,-.5) to (1.9,1.1);
\draw [-,gray,dotted,thick] (1.9,1.1) to (2,.6);
  \end{tikzpicture}
  \\
  \hline
  $1$ &
   \begin{tikzpicture}[scale=1]
         \draw [-,blue,thick] (0.5,.5) to (1,0);
        \draw [fill=black] (0.5,0.5) circle [radius=0.03];
             \draw [fill=black] (1,0) circle [radius=0.03];
         \draw [-,gray,dotted,thick] (1,0) to (1.25,.25);
        \draw [-,gray,dotted,thick] (1.25,.25) to (1.5,0);
          \draw [fill=white] (1.25,-0.2) circle [radius=0.0];
                  \draw [fill=white] (1.25,0.8) circle [radius=0.0];
         \draw [-,gray,dotted,thick] (1.5,0) to (1.75,-1);
\draw [-,gray,dotted,thick] (1.75,-1) to (1.9,.6);
\draw [-,gray,dotted,thick] (1.9,.6) to (2,.1);
  \end{tikzpicture}
  &
  \begin{tikzpicture}[scale=1]
        \draw [-,blue,thick] (0.5,0) to (1,0);
         \draw [-,gray,dotted,thick] (1,0) to (1.25,.25);
        \draw [-,gray,dotted,thick] (1.25,.25) to (1.5,0);

        \draw [fill=black] (1.0,0) circle [radius=0.03];
             \draw [fill=black] (0.5,0) circle [radius=0.03];
             \draw [fill=white] (1.25,-0.2) circle [radius=0.0];
                  \draw [fill=white] (1.25,0.8) circle [radius=0.0];
         \draw [-,gray,dotted,thick] (1.5,0) to (1.75,-1);
\draw [-,gray,dotted,thick] (1.75,-1) to (1.9,.6);
\draw [-,gray,dotted,thick] (1.9,.6) to (2,.1);
  \end{tikzpicture}
   \\  \hline\hline
  \end{tabular}
  \caption{ The fixed part of the curve $\vv \gamma\begin{pmatrix}
       \tau' &\vv \tau
    \\ \xi' & \vv \xi
  \end{pmatrix}$  is indicated by the dotted lines, which generically can  go below and above the starting point $\gamma_0=0$ of the entire curve. (The dotted lines are drawn not up to scale.) The initial segment that depends on the values of  $\tau',\xi'$ is marked by a thick solid line.
    \label{TabLeft}}
\end{table}

Thus if $m<0$ or $\xi'=1$, formula \eqref{Qa} is equivalent to
\begin{equation}
  \label{Q12prove}
  \wt w_{\vv \sigma\topp 0}(z)-q \wt w_{\vv\sigma\topp 1}(z)=\wt w_{\sigma}(z),
\end{equation}
where
\begin{equation}
  \label{def-sigma}
  \vv \sigma\topp {\tau'}=\sigma\left(\gamma\begin{pmatrix}
  \tau' & \vv \tau \\ \xi' & \vv \xi
\end{pmatrix} \right),\quad \tau'=0,1.
\end{equation}
Of course compositions $\vv\sigma\topp 0$, $\vv \sigma\topp 1$ depend also  on $\xi'\in\{0,1\}$ which is fixed in \eqref{Q12prove}.

We note that in the exceptional case $m=0$ and $\xi'=0$, where \eqref{g1L} has a different form, formula \eqref{Qa} is equivalent to
\begin{equation} 
  \label{Qexcept}
    \wt w_{\vv \sigma\topp 0}(z)-q  z \wt w_{\vv\sigma\topp 1}(z)=\wt w_{\sigma}(z)
\end{equation}
instead of \eqref{Q12prove}.

We now prove \eqref{Q12prove} in the generic case $m<0$ and $M>0$.
By going over the cases $\xi'=0,1$, it is straightforward to check that in the generic case,
the number of parts in the compositions $\vv \sigma$, $\vv \sigma\topp 0$ and $\vv \sigma\topp 1$ is the same, and the only part of the composition that changes is the one that counts the number of crossings through $0$.
Recall that $\sigma_0(\vv \gamma)$ is the number of times $m=\min(\vv \gamma)$ is attained,
so the part of $\vv \sigma$ that counts the number of times that the initial level $0$ is attained has index $J_0=-m$.
 For the other two compositions, the index of the part that counts crossings of 0  is $J'= -m+\xi '-1
 =J_0+\xi'-1$ for $\vv\sigma\topp 1$ and $J'+1=-m+\xi'$ for $\vv\sigma\topp 0$, see Table \ref{TabLeft}.
 We get 
\begin{equation}
  \label{Left-J'-generic}
  \sigma\topp 0_j=\begin{cases}
1+\sigma_j & j = J'+1, \\
  \sigma_j & \mbox { otherwise},
\end{cases}
\quad \mbox{ and } \quad \sigma\topp 1_j=\begin{cases}
1+\sigma_j & j = J', \\
  \sigma_j & \mbox {otherwise},
\end{cases}
\end{equation}
where $j=0,\dots,r$.
\arxiv{
If $\xi'=0$ then $\sigma\topp 0_{-m}=\sigma_{-m}+1$. If $\xi'=1$ and $m<0$ then $\sigma\topp 0_{0}=\sigma_0$ counts the number of crossings of the level $m-1$ by the new path, so the number of crossings of level $0$ increases by 1 and $\sigma\topp 0_{-m+1}=\sigma_{-m+1}+1$. The indexes of parts that increase are thus $J'+1=-m+\xi'$ as claimed. (This however requires $M>0$ to ensure $\sigma_{-m+1}>0$ and requires $m<0$ in order for the minimum not to be in the first step; otherwise $\sigma\topp 0_0=1$ and the number of parts of the composition increases by 1.)

Similar reasoning applies to $\vv \sigma\topp 1$: If $\xi'=0$ then
the new minimum is $m+1$, provided that $m<0$ and $\sigma\topp 1_{-m-1}=\sigma_{-m-1}+1$. Thus the increase in composition is at part with index $J'=-m+\xi'-1$. (This is where the assumption $m<0$ is used: if $m=0$ and $\xi'=0$ then the new curve has a minimum $0$ uniquely obtained at the origin $\sigma\topp 1_0=1$, and the number of parts increases.)

For $\xi'=1$,  the minimum is unchanged and $\sigma\topp 1_{-m}=\sigma_{-m}+1$, so the index of the part that increases is $J'=-m+\xi'-1$ as claimed.
}
Therefore, with
\begin{eqnarray}
\calX&:=&(\Dq z)^{\sigma_0-1}\Dq(\Dq z)^{\sigma_1-1}\Dq\dots (\Dq z)^{ \sigma_{J'-1}-1}\Dq \label{calX}
\\ \label{calY}
\calY&:=& (\Dq z)^{\sigma_{J'+2}-1}\Dq\dots (\Dq z)^{\sigma_r-1}\Dq\\
 \label{calU0} \calU&:=&(\Dq z)^{ \sigma_{J'}-1}\Dq(\Dq z)^{ \sigma_{J'+1}}\Dq \\
 \label{calU1}  \calV &:=&(\Dq z)^{ \sigma_{J'}}\Dq(\Dq z)^{ \sigma_{J'+1}-1}\Dq,
\end{eqnarray}
from \eqref{w2Dq'} we get
$$\wt w_{\vv \sigma\topp 0}(x)=\calX \calU\calY \left[\tfrac{1}{1-z}\right], \quad \wt w_{\vv\sigma\topp 1}(x)=\calX \calV \calY \left[\tfrac{1}{1-z}\right].$$
Since
$$\calU = (\Dq z)^{\sigma_J'-1}\Dq\Dq z (\Dq z)^{\sigma_{J'+1}-1}\Dq$$
and
$$
\calV= (\Dq z)^{\sigma_J'-1}\Dq z \Dq (\Dq z)^{\sigma_{J'+1}-1}\Dq,
$$
factoring out the common factors we get
\begin{equation}
  \label{U-comm}
  \calU-q\,\calV= (\Dq z)^{ \sigma_{J'}-1}\Dq (\Dq z-q\, z \Dq)
(\Dq z)^{ \sigma_{J'+1}-1}\Dq=(\Dq z)^{ \sigma_{J'}-1}\Dq
(\Dq z)^{ \sigma_{J'+1}-1}\Dq
\end{equation}
by \eqref{D-q-comm}.
This ends the proof of \eqref{Q12prove} when $m<0$ and $M>0$.

Most special cases with $m=0$ or $M=0$ require minor changes to the above argument and are omitted. We present one such argument for the case $m=0=M$ and $\xi'=1$.
\subsubsection{Case $m=M=0$ and $\xi'=1$}\label{Sec:M=m=0}
In this case, composition $\vv \sigma=(L+1)$ has only one part  ($r=M-m=0$), $\vv \sigma\topp 0=(L+1,1)$ has two parts, and
$\vv \sigma\topp 1=(L+2)$ has one part. Since
$$\Dq[\tfrac{1}{1-z}]=\Dq\left[1+\tfrac{z}{1-z}\right]=(\Dq z) [\tfrac{1}{1-z}] ,$$ representation \eqref{w2Dq'} gives
$$
\wt w_{\vv \sigma\topp 0}(z)=(\Dq z)^{L}\Dq \Dq [\tfrac{1}{1-z}] = \left((\Dq z)^{L}\Dq\right) (\Dq z) [\tfrac{1}{1-z}]
, $$
$$
\wt w_{\vv\sigma\topp 1}(z)=(\Dq z)^{L+1}\Dq [\tfrac{1}{1-z}] = \left((\Dq z)^{L}\Dq\right) (z \Dq) [\tfrac{1}{1-z}].
$$
So  \eqref{Q12prove} again follows from \eqref{D-q-comm}.

\arxiv{
Here are the omitted details for the remaining special cases. %

We first note that if $m=0$ and $\xi'=0$ then to prove \eqref{Qa} we need to show that
\eqref{Qexcept} is true.
In all other cases, we need to prove \eqref{Q12prove}.
\subsubsection{Case $m=0$, $M>0$}\label{Sec:m=0;M>0}
If $\xi'=0$ then from Table \ref{TabLeft} we see that
\begin{enumerate}
  [(a)]
\item  composition  $\vv\sigma\topp 0$ of $L+2$ has $r+1$ parts with $\sigma_0\topp 0=1+\sigma_0$ and
$\sigma_j\topp 0=\sigma_j$ for $j=1,\dots,r$.
\item composition $\vv \sigma\topp 1$ of $L+2$  has $r+2$ parts with $\sigma\topp 1_0=1$ and
$\sigma_j\topp 1=\sigma_{j-1}$ for $j=1,\dots,r+1$.
\end{enumerate}
From \eqref{w2Dq'} we see that
$$
\wt w_{\vv \sigma\topp 0}(z) =(\Dq z) w_{\vv \sigma}(z) \mbox{ and }
\wt w_{\vv\sigma\topp 1}(z) = \Dq w_{\vv \sigma}(z)
$$
 and \eqref{Qexcept} follows by \eqref{D-q-comm}.

 Note that this argument applies also to the case $M=0$,
 so taking into account Section \ref{Sec:M=m=0},  the proof for the case $M=m=0$ is now complete.

 On the other hand, if $\xi'=1$ and $M>0$ then $r>0$, $\vv\sigma \topp 0$ and $\vv\sigma\topp 1$ are both compositions of $L+2$ with $r+1$ parts. We have
 $\sigma\topp 0_0=\sigma_0$, $\sigma\topp 0_1=1+\sigma_1$, $\sigma\topp 0_j=\sigma_j$ for $j=2,\dots,r$ while $\sigma\topp 1_0=1+\sigma_0$ and $\sigma\topp 1_j=\sigma_j$ for $j=1,\dots r$. This  matches the formulas \eqref{Left-J'-generic} for the generic case with $J'=0$, so \eqref{Q12prove} follows.
\subsubsection{Case $m<0$, $M=0$}\label{Sec:m<0M-0}
This is another case where the number of parts in the compositions may change. Note that $\vv \sigma$ has $r+1$ parts with $r=-m$.
If $\xi'=0$ then, see Table \ref{TabLeft}, then $\sigma_{r}\topp 0=1+\sigma_{r}$ and $\sigma\topp 0_j=\sigma_j$ for $j=0,\dots,r-1=r-1$. On the other hand,
$\sigma_{r-1}\topp 1=\sigma_{r-1}$ while $\sigma_j\topp 1=\sigma_j$ for the other $j\in\{0,\dots,r\}$. Thus, the three compositions have $r+1$ parts and the formulas \eqref{Left-J'-generic} for the generic case with $J'=r-1=-m-1$ hold, so \eqref{Q12prove} follows.

If $\xi'=1$ then the number of parts in $\vv \sigma\topp 0$ increases to $r+2$. We have
$\sigma\topp 0_{r+1}=1$, and $\sigma\topp 0_j=\sigma_j$ for $j=0,\dots,r$.
Clearly, $\sigma\topp 1_j=\sigma_j$ for $j=0,\dots,r-1$ and $\sigma\topp 1_{r}=1+\sigma_r$.
Writing $\wt w_{\vv \sigma}(z)=w_{\sigma}(z)/(z;q)_{L+2}$ in the operator form as $\calY \Dq[(1-z)^{-1}]$ as in \eqref{w2Dq'} we see that
$$
\wt w_{\vv \sigma\topp 0}(z)=\calY \Dq \Dq [\tfrac{1}{1-z}]
=\calY \Dq \Dq  \left[1+\tfrac{z}{1-z}\right]
=\calY \Dq (\Dq z) [\tfrac{1}{1-z}]
$$
and
$$
\wt w_{\vv \sigma\topp 0}(z)=\calY (\Dq z)\Dq [\tfrac{1}{1-z}] =
\calY \Dq (z\Dq) [\tfrac{1}{1-z}].
$$
 Thus \eqref{Q12prove} follows by \eqref{D-q-comm} again.
}
\subsubsection{Right boundary}
Next, we prove \eqref{Qb}. With $\vv \gamma=\gamma\begin{pmatrix}
     \vv \tau
    \\   \vv \xi
  \end{pmatrix}$ of length $L$, let $m=\min \vv \gamma$ and $M=\max \vv \gamma$. The {\em generic} case to consider first is $m<\gamma_L$ and $M>\gamma_L$.

Inspecting the cases listed in Table \ref{TabRight}, we see that
\begin{equation*}
  \label{g-right1}
  g_{\A,\B} \begin{pmatrix}
     \vv \tau &1
    \\   \vv \xi & \xi'
  \end{pmatrix} =  \B^{1-\xi '} g_{\A,\B}\begin{pmatrix}
     \vv \tau
    \\   \vv \xi
  \end{pmatrix}
\end{equation*}
and
\begin{equation}
  \label{g-right0}
  g_{\A,\B}\begin{pmatrix}
     \vv \tau & 0
    \\   \vv \xi & \xi'
  \end{pmatrix} =  \begin{cases}
    \A\, g_{\A,\B}\begin{pmatrix}
     \vv \tau
    \\   \vv \xi
  \end{pmatrix} & \mbox{ if } m = \gamma_L \mbox{ and  } \xi'=1, \\
   \B^{-\xi '} g_{\A,\B}\begin{pmatrix}
     \vv \tau
    \\   \vv \xi
  \end{pmatrix} & \mbox{otherwise}.
  \end{cases}
\end{equation}
(Recall that $\B>0$ throughout this proof.)
\begin{table}[htb]
  \begin{tabular}{||c||c|c||} \hline\hline
  $\xi'$ & $\vv \gamma\begin{pmatrix}
     \vv \tau &1
    \\   \vv \xi & \xi'
  \end{pmatrix}$ & $\vv \gamma\begin{pmatrix}
     \vv \tau & 0
    \\   \vv \xi & \xi'
  \end{pmatrix}$ \\
  \hline\hline
  $0$ &  \begin{tikzpicture}[scale=1]
     \draw [-,gray,dotted,thick] (-1,.8) to (-.75,-.5);
   \draw [-,gray,dotted,thick] (-.75,-.5) to (-.5,.25);
    \draw [-,gray,dotted,thick] (-.5,.25) to (-.25,.0);
     \draw [-,gray,dotted,thick] (-.25,.0) to (0,0);
       \draw [-,blue,thick] (0,0) to (0.5,.5);
                    \draw [fill=black] (0.5,0.5) circle [radius=0.03];
             \draw [fill=black] (0,0) circle [radius=0.03];
                 \draw [fill=white] (1.25,0.85) circle [radius=0.0];
  \end{tikzpicture}
  &

  \begin{tikzpicture}[scale=1]
       \draw [-,gray,dotted,thick] (-1,.8) to (-.75,-.5);
   \draw [-,gray,dotted,thick] (-.75,-.5) to (-.5,.25);
    \draw [-,gray,dotted,thick] (-.5,.25) to (-.25,.0);
     \draw [-,gray,dotted,thick] (-.25,.0) to (0,0);
       \draw [-,blue,thick] (0,0) to (0.5,0);
                    \draw [fill=black] (0.5,0) circle [radius=0.03];
             \draw [fill=black] (0,0) circle [radius=0.03];

                 \draw [fill=white] (1.25,0.85) circle [radius=0.0];
  \end{tikzpicture}
  \\
  \hline
  $1$ &
   \begin{tikzpicture}[scale=1]
        \draw [-,gray,dotted,thick] (-1,.8) to (-.75,-.5);
   \draw [-,gray,dotted,thick] (-.75,-.5) to (-.5,.25);
    \draw [-,gray,dotted,thick] (-.5,.25) to (-.25,.0);
     \draw [-,gray,dotted,thick] (-.25,.0) to (0,0);
       \draw [-,blue,thick] (0,0) to (0.5,0);

             \draw [fill=black] (0.5,0) circle [radius=0.03];
        \draw [fill=black] (0.0,0) circle [radius=0.03];

  \end{tikzpicture}
  &

  \begin{tikzpicture}[scale=1]
       \draw [-,gray,dotted,thick] (-1,.8) to (-.75,-.5);
   \draw [-,gray,dotted,thick] (-.75,-.5) to (-.5,.25);
    \draw [-,gray,dotted,thick] (-.5,.25) to (-.25,.0);
     \draw [-,gray,dotted,thick] (-.25,.0) to (0,0);
       \draw [-,blue,thick] (0,0) to (0.5,-.5);
        \draw [fill=black] (0.5,-.5) circle [radius=0.03];
             \draw [fill=black] (0,0) circle [radius=0.03];
                 \draw [fill=white] (1.25,0.45) circle [radius=0.0];
  \end{tikzpicture}
   \\  \hline\hline
  \end{tabular}
  \caption{ The fixed part of the curve $\vv \gamma\begin{pmatrix}
       \vv \tau& \tau'
    \\ \vv \xi & \xi'
  \end{pmatrix}$  is indicated by the dotted lines,
  which generically can  go below and above the end point $\gamma_L$ of the curve. (Doted lines are drawn not to scale.) The last segment that depends on the values of  $\tau',\xi'$ is marked by the thick solid line.\label{TabRight}}
\end{table}

\arxiv{Indeed, if $\tau'=1$ or in the generic case,
$$\min \gamma \begin{pmatrix}
     \vv \tau &1
    \\   \vv \xi & \xi'
  \end{pmatrix}= \min \gamma\begin{pmatrix}
     \vv \tau
    \\   \vv \xi
  \end{pmatrix},$$
  while the end-point is at $\gamma_L+\tau'-\xi'$.
  }
Thus  if $m<\gamma_L$ and $M>\gamma_L$ then \eqref{Qb} reduces to the identity
\begin{equation}
  \label{Qb2prove}
  \wt w_{\vv\sigma\topp 1}(z)-q \wt w_{\vv \sigma\topp 0}(z)=\wt w_{\sigma}(z)
\end{equation}
  with
$\vv \sigma\topp {\tau'}=\sigma\left(\gamma\begin{pmatrix}
   \vv \tau &\tau'  \\ \vv \xi &\xi'
\end{pmatrix} \right)$, where $\xi'\in\{0,1\}$ is fixed.
(This is   a right boundary  analog  of \eqref{Q12prove} and \eqref{def-sigma}.)

We note that in the exceptional case $m=\gamma_L$ and $\xi'=1$ where \eqref{g-right0} has a different form,  to deduce equation \eqref{Qb} we need 
\begin{equation}
  \label{w-exc-right}
    \wt w_{\vv\sigma\topp 1}(z)-q\, z \wt w_{\vv \sigma\topp 0}(z)=\wt w_{\sigma}(z)
\end{equation}
instead of \eqref{Qb2prove}. (The details of this case are omitted.)

Recall that $\vv \sigma$ is a composition of $L+1$ into $r+1$ parts that counts the contributions \eqref{def:sigma} of the values of $\vv \gamma=\gamma\begin{pmatrix}
  \vv \tau\\\vv \xi
\end{pmatrix}$, so that $r=\max\vv\gamma-\min\vv\gamma$. Let $J_0$ be the index of the part of $\vv\sigma$ that counts the contribution of the end-point $\gamma_L$ of the path. In the generic case $M>\gamma_L$ and $m<\gamma_L$, so we have $0<J_0<r+1$.

The compositions $\vv \sigma\topp 1$ and $\vv \sigma\topp 0$ differ from $\vv \sigma$ only in the part that counts the contribution of the end point $\gamma_L+\tau'-\xi'$ of the path. Since the contribution of the minimum has index 0, see \eqref{def:sigma}, the index of the part that changes is
$J_0+\tau'-\xi'$.
Inspecting Table \ref{TabRight}, we see that in the generic case both $\vv \sigma\topp1$ and $\vv \sigma\topp 0$ are compositions of $L+2$ into $r+1$ parts, with exactly one part that increases:
\begin{equation} \label{gen2a}
  \sigma\topp 1_j= \begin{cases}
    1+\sigma_j & j=J'+1,  \\
    \sigma_j & \mbox{otherwise},
  \end{cases}
\end{equation}
\begin{equation} \label{gen2b}
  \sigma\topp 0_j= \begin{cases}
    1+\sigma_j & j=J',  \\
    \sigma_j & \mbox{otherwise,}
  \end{cases}
\end{equation}
where $J'=J_0-\xi'$.

Comparing this with the compositions that appeared in the proof of \eqref{Q12prove}, we see that
$\wt w_{\vv\sigma\topp 1}=\calX\calU\calY[1/(1-z)]$,  $\wt w_{\vv \sigma\topp 0}=\calX\calV\calY[1/(1-z)]$ with the same operators (\ref{calX}-\ref{calU1}) that we used before, so \eqref{Qb2prove} follows by the previous argument, see \eqref{U-comm}.

The remaining cases $m=\gamma_L$ or $M=\gamma_L$ require some changes and are omitted.
\arxiv{
Here are the omitted details. First, we note that in the exceptional case $m=\gamma_L$ and $\xi'=1$ expression \eqref{Qb} reduces to  \eqref{w-exc-right}
 instead of \eqref{Qb2prove}. In all other cases, we need to prove \eqref{Qb2prove}. In the proof we rely on the graphs  in Table \ref{TabRight}.
 \subsubsection{Case $m=\gamma_L$, $M>m$}\label{Sec226}
 If $\xi'=1$ then $\sigma\topp 1_0=1+\sigma_0$ and $\sigma\topp 1_j=\sigma_j$ for $j=1,\dots,r$.
 On the other hand, $\sigma\topp 0_0=1$ and $\sigma\topp 0_{j}=\sigma_{j-1}$ for $j=1, \dots r+1$.
 Thus writing $\wt w_{\vv \sigma}(z)=(\Dq z)^{\sigma_0-1}\Dq \calY [\tfrac1{1-z}]$, we have
 $$
 \wt w_{\vv\sigma\topp 1}(z)-q z  \wt w_{\sigma_0}(z)=
 (\Dq z)^{\sigma_0}\Dq \calY [\tfrac1{1-z}] - q  z \Dq (\Dq z)^{\sigma_0-1}\Dq \calY [\tfrac1{1-z}]
 =(\Dq z - z  \Dq)\wt w_{\vv \sigma}(z).
 $$
   By \eqref{D-q-comm}, this proves \eqref{w-exc-right}, which is what is needed in the exceptional case of \eqref{g-right0}.

 If $\xi'=0$ then $\sigma_1\topp 1=1+\sigma_1$ and $\sigma\topp 1_j=\sigma_j$ for the other $j\in\{0,\dots,r\}$.  On the other hand, $\sigma\topp 0_0=1+\sigma_0$ and $\sigma\topp 0_j=\sigma_j$ for $j=1,\dots,r$. This matches the "generic case" \eqref{gen2a}, \eqref{gen2b} with $J'=0$, so \eqref{Qb2prove} follows.
}
\arxiv{
  \subsubsection{Case $M=\gamma_L$, $M>m$}\label{Sec227}
  In this case we have again compositions into different number of parts. Note that $\gamma_L$ contributes to the last part $\sigma_r$ of composition $\vv \sigma$.
  \begin{enumerate}[(a)]
    \item If $\xi'=0$ then $\vv\sigma\topp 1$ has $r+2$ parts with $\sigma\topp 1_j=\sigma_j$ for $j=0,\dots,r$ and $\sigma\topp 1_{r+1}=1$. On the other hand $\sigma\topp 0_j=\sigma_j$ for $j=0,\dots,r-1$ and $\sigma\topp 0_r=1+\sigma_r$.
        Thus, as in Section \ref{Sec:m<0M-0} we get
        \begin{equation}
            \label{227a}
             \wt w_{\vv\sigma\topp 1}(z)=\calY \Dq \Dq [\tfrac{1}{1-z}] =
        \calY \Dq (\Dq z) [\tfrac{1}{1-z}], \mbox{ and }
         \wt w_{\vv\sigma\topp 1}(z)= \calY \Dq (z \Dq)[\tfrac{1}{1-z}]
        \end{equation}

        and \eqref{Qb2prove} follows from \eqref{D-q-comm}.
  \item If $\xi'=1$ then both $\vv \sigma\topp 1$ and $\vv \sigma\topp 0$  have $r+1$ parts. We have $\sigma\topp 1_j=\sigma_j$ for $j=0,\dots,r-1$ and $\sigma\topp 1_{r}=1+\sigma_r$. For the second composition, we get $\sigma\topp 0_{r-1}=1+\sigma_r-1$ and $\sigma\topp 0_j=\sigma_j$ for all other $j\in\{0,\dots r\}$. This coincides again with the "generic case" \eqref{gen2a}, \eqref{gen2b} with $J'=r-1$, so \eqref{Qb2prove} follows.
  \end{enumerate}
    \subsubsection{Case $M=m=\gamma_L$}
    This is the case $M=m=0$ with  $\vv \sigma=(L+1)$. This case is handled  by listing all the compositions explicitly, as in Section \ref{Sec:M=m=0}.

    If $\xi'=0$ then $\vv \sigma\topp 1 =(L+1,1)$ and $\sigma\topp 0 =(L+2)$. So, the representation \eqref{227a} holds with  $\calY= (\Dq z)^{L}$, and \eqref{Qb2prove} follows as in Section \ref{Sec227}.

    If $\xi'=1$, we are in the exceptional case where we need to verify \eqref{w-exc-right}. We have  $\vv \sigma\topp 1 =(L+2)$ and  $\vv \sigma\topp 0 =(1,L+1)$. Thus as in Section \ref{Sec226} with $\sigma_0=L+1$, we have
    $$\wt w_{\vv\sigma\topp 1}(z)-q\, z \wt w_{\vv \sigma\topp 0}(z)=\left((\Dq z)^{L+1}\Dq - q z \Dq(\Dq z)^{L}\Dq\right)[\tfrac1{1-z}]
    = \left(\Dq z - q z \Dq\right)(\Dq z)^{L}\Dq[\tfrac1{1-z}].$$
    So \eqref{w-exc-right} follows from \eqref{D-q-comm}.
}

\begin{table}[htb]
  \begin{tabular}{||c||c|c|c||}\hline\hline
  $(\xi', \xi'')$ & $\vv \gamma\begin{pmatrix}
    \vv \tau_1 & 1 & 0 &\vv \tau_2
    \\ \vv \xi_1 & \xi' & \xi'' & \vv \xi_2
  \end{pmatrix}$ & $\vv \gamma\begin{pmatrix}[.5]
    \vv \tau_1 & 0 & 1 &\vv \tau_2
    \\ \vv \xi_1 & \xi' & \xi'' & \vv \xi_2
  \end{pmatrix}$ & $\vv \gamma\begin{pmatrix}[.5]
    \vv \tau_1 & 1-\xi''  &\vv \tau_2
    \\ \vv \xi_1 & \xi' &  \vv \xi_2
  \end{pmatrix}$ \\
  \hline\hline
  $(0,0)$ &  \begin{tikzpicture}[scale=1]
       \draw [-,gray,dotted,thick] (-1,.8) to (-.75,-.5);
   \draw [-,gray,dotted,thick] (-.75,-.5) to (-.5,.25);
    \draw [-,gray,dotted,thick] (-.5,.25) to (-.25,.0);
     \draw [-,gray,dotted,thick] (-.25,.0) to (0,0);
       \draw [-,blue,thick] (0,0) to (0.5,.5);
        \draw [-,blue,thick] (0.5,0.5) to (1,.5);
                    \draw [fill=black] (0.5,0.5) circle [radius=0.03];
        \draw [fill=black] (1.0,.5) circle [radius=0.03];
             \draw [fill=black] (0,0) circle [radius=0.03];
         \draw [-,gray,dotted,thick] (1,.5) to (1.25,.75);
        \draw [-,gray,dotted,thick] (1.25,.75) to (1.5,.5);

                \draw [-,gray,dotted,thick] (1.5,.5) to (1.75,-.5);
              \draw [-,gray,dotted,thick] (1.75,-.5) to (1.9,1.1);
\draw [-,gray,dotted,thick] (1.9,1.1) to (2,.6);
                 \draw [fill=white] (1.25,0.85) circle [radius=0.0];
  \end{tikzpicture}
  &

  \begin{tikzpicture}[scale=1]
       \draw [-,gray,dotted,thick] (-1,.8) to (-.75,-.5);
   \draw [-,gray,dotted,thick] (-.75,-.5) to (-.5,.25);
    \draw [-,gray,dotted,thick] (-.5,.25) to (-.25,.0);
     \draw [-,gray,dotted,thick] (-.25,.0) to (0,0);
       \draw [-,blue,thick] (0,0) to (0.5,0);
        \draw [-,blue,thick] (0.5,0) to (1,.5);
                    \draw [fill=black] (0.5,0) circle [radius=0.03];
        \draw [fill=black] (1.0,.5) circle [radius=0.03];
             \draw [fill=black] (0,0) circle [radius=0.03];
         \draw [-,gray,dotted,thick] (1,.5) to (1.25,.75);
        \draw [-,gray,dotted,thick] (1.25,.75) to (1.5,.5);
                         \draw [-,gray,dotted,thick] (1.5,.5) to (1.75,-.5);
\draw [-,gray,dotted,thick] (1.75,-.5) to (1.9,1.1);
\draw [-,gray,dotted,thick] (1.9,1.1) to (2,.6);
                 \draw [fill=white] (1.25,0.85) circle [radius=0.0];
  \end{tikzpicture}
  &
    \begin{tikzpicture}[scale=1]
         \draw [-,gray,dotted,thick] (-.5,.8) to (-.25,-.5);
   \draw [-,gray,dotted,thick] (-.25,-.5) to (0,.25);
    \draw [-,gray,dotted,thick] (0,.25) to (.25,.0);
     \draw [-,gray,dotted,thick] (.25,.0) to (.5,0);
        \draw [-,blue,thick] (0.5,0.0) to (1,.5);
         \draw [-,gray,dotted,thick] (1,.5) to (1.25,.75);
        \draw [-,gray,dotted,thick] (1.25,.75) to (1.5,.5);
                            \draw [fill=black] (0.5,0) circle [radius=0.03];
        \draw [fill=black] (1.0,.5) circle [radius=0.03];
                         \draw [-,gray,dotted,thick] (1.5,.5) to (1.75,-.5);
\draw [-,gray,dotted,thick] (1.75,-.5) to (1.9,1.1);
\draw [-,gray,dotted,thick] (1.9,1.1) to (2,.6);
                 \draw [fill=white] (1.25,0.85) circle [radius=0.0];
  \end{tikzpicture}
  \\
  \hline
  $(1,0)$ &
   \begin{tikzpicture}[scale=1]
        \draw [-,gray,dotted,thick] (-1,.8) to (-.75,-.5);
   \draw [-,gray,dotted,thick] (-.75,-.5) to (-.5,.25);
    \draw [-,gray,dotted,thick] (-.5,.25) to (-.25,.0);
     \draw [-,gray,dotted,thick] (-.25,.0) to (0,0);
       \draw [-,blue,thick] (0,0) to (0.5,0);

        \draw [-,blue,thick] (0.5,0) to (1,0);
             \draw [fill=black] (0.5,0) circle [radius=0.03];
        \draw [fill=black] (0.0,0) circle [radius=0.03];
             \draw [fill=black] (1,0) circle [radius=0.03];
        \draw [fill=black] (0.0,0) circle [radius=0.03];
         \draw [-,gray,dotted,thick] (1,0) to (1.25,.25);
        \draw [-,gray,dotted,thick] (1.25,.25) to (1.5,0);
                 \draw [-,gray,dotted,thick] (1.5,0) to (1.75,-1);
\draw [-,gray,dotted,thick] (1.75,-1) to (1.9,.6);
\draw [-,gray,dotted,thick] (1.9,.6) to (2,.1);
  \end{tikzpicture}
  &

  \begin{tikzpicture}[scale=1]
       \draw [-,gray,dotted,thick] (-1,.8) to (-.75,-.5);
   \draw [-,gray,dotted,thick] (-.75,-.5) to (-.5,.25);
    \draw [-,gray,dotted,thick] (-.5,.25) to (-.25,.0);
     \draw [-,gray,dotted,thick] (-.25,.0) to (0,0);
       \draw [-,blue,thick] (0,0) to (0.5,-.5);
        \draw [-,blue,thick] (0.5,-.5) to (1,0);
         \draw [-,gray,dotted,thick] (1,0) to (1.25,.25);
        \draw [-,gray,dotted,thick] (1.25,.25) to (1.5,0);

        \draw [fill=black] (0.5,-.5) circle [radius=0.03];
        \draw [fill=black] (1.0,0) circle [radius=0.03];
             \draw [fill=black] (0,0) circle [radius=0.03];

                 \draw [-,gray,dotted,thick] (1.5,0) to (1.75,-1);
\draw [-,gray,dotted,thick] (1.75,-1) to (1.9,.6);
\draw [-,gray,dotted,thick] (1.9,.6) to (2,.1);
                 \draw [fill=white] (1.25,0.45) circle [radius=0.0];
  \end{tikzpicture}
  &
    \begin{tikzpicture}[scale=1]
         \draw [-,gray,dotted,thick] (-.5,.8) to (-.25,-.5);
   \draw [-,gray,dotted,thick] (-.25,-.5) to (0,.25);
    \draw [-,gray,dotted,thick] (0,.25) to (.25,.0);
     \draw [-,gray,dotted,thick] (.25,.0) to (.5,0);
        \draw [-,blue,thick] (0.5,0.0) to (1,.0);
          \draw [fill=black] (0.5,0) circle [radius=0.03];
            \draw [fill=black] (1,0) circle [radius=0.03];
         \draw [-,gray,dotted,thick] (1,.0) to (1.25,.25);
        \draw [-,gray,dotted,thick] (1.25,.25) to (1.5,0);
                       \draw [-,gray,dotted,thick] (1.5,0) to (1.75,-1);
\draw [-,gray,dotted,thick] (1.75,-1) to (1.9,.6);
\draw [-,gray,dotted,thick] (1.9,.6) to (2,.1);
  \end{tikzpicture}
  \\ \hline
  $(0,1)$ &
    \begin{tikzpicture}[scale=1]
         \draw [-,gray,dotted,thick] (-1,.8) to (-.75,-.5);
   \draw [-,gray,dotted,thick] (-.75,-.5) to (-.5,.25);
    \draw [-,gray,dotted,thick] (-.5,.25) to (-.25,.0);
     \draw [-,gray,dotted,thick] (-.25,.0) to (0,0);
       \draw [-,blue,thick] (0,0) to (0.5,.5);
        \draw [-,blue,thick] (0.5,.5) to (1,0);
         \draw [-,gray,dotted,thick] (1,0) to (1.25,.25);
        \draw [-,gray,dotted,thick] (1.25,.25) to (1.5,0);
                         \draw [-,gray,dotted,thick] (1.5,0) to (1.75,-1);
\draw [-,gray,dotted,thick] (1.75,-1) to (1.9,.6);
\draw [-,gray,dotted,thick] (1.9,.6) to (2,.1);
         \draw [fill=black] (0.5,0.5) circle [radius=0.03];
         \draw [fill=white] (0.5,0.8) circle [radius=0.0];
        \draw [fill=black] (1.0,0) circle [radius=0.03];
             \draw [fill=black] (0,0) circle [radius=0.03];
  \end{tikzpicture}
  &
   \begin{tikzpicture}[scale=1]
        \draw [-,gray,dotted,thick] (-1,.8) to (-.75,-.5);
   \draw [-,gray,dotted,thick] (-.75,-.5) to (-.5,.25);
    \draw [-,gray,dotted,thick] (-.5,.25) to (-.25,.0);
     \draw [-,gray,dotted,thick] (-.25,.0) to (0,0);
       \draw [-,blue,thick] (0,0) to (0.5,0);
       \draw [fill=black] (0.5,0) circle [radius=0.03];
        \draw [-,blue,thick] (0.5,0) to (1,0);
         \draw [-,gray,dotted,thick] (1,0) to (1.25,.25);
        \draw [-,gray,dotted,thick] (1.25,.25) to (1.5,0);
                 \draw [-,gray,dotted,thick] (1.5,0) to (1.75,-1);
\draw [-,gray,dotted,thick] (1.75,-1) to (1.9,.6);
\draw [-,gray,dotted,thick] (1.9,.6) to (2,.1);
                \draw [fill=black] (0.5,0.0) circle [radius=0.03];
        \draw [fill=black] (1.0,0) circle [radius=0.03];
             \draw [fill=black] (0,0) circle [radius=0.03];
  \end{tikzpicture}
  &
   \begin{tikzpicture}[scale=1]
           \draw [-,gray,dotted,thick] (-.5,.8) to (-.25,-.5);
   \draw [-,gray,dotted,thick] (-.25,-.5) to (0,.25);
    \draw [-,gray,dotted,thick] (0,.25) to (.25,.0);
     \draw [-,gray,dotted,thick] (.25,.0) to (.5,0);
        \draw [-,blue,thick] (0.5,0.0) to (1,.0);
          \draw [fill=black] (0.5,0) circle [radius=0.03];
            \draw [fill=black] (1,0) circle [radius=0.03];
         \draw [-,gray,dotted,thick] (1,.0) to (1.25,.25);
        \draw [-,gray,dotted,thick] (1.25,.25) to (1.5,0);
                 \draw [-,gray,dotted,thick] (1.5,0) to (1.75,-1);
\draw [-,gray,dotted,thick] (1.75,-1) to (1.9,.6);
\draw [-,gray,dotted,thick] (1.9,.6) to (2,.1);

  \end{tikzpicture}
  \\ \hline
  $(1,1)$ &
   \begin{tikzpicture}[scale=1]
          \draw [-,gray,dotted,thick] (-1,.8) to (-.75,-.5);
   \draw [-,gray,dotted,thick] (-.75,-.5) to (-.5,.25);
    \draw [-,gray,dotted,thick] (-.5,.25) to (-.25,.0);
     \draw [-,gray,dotted,thick] (-.25,.0) to (0,0);
       \draw [-,blue,thick] (0,0) to (0.5,.0);
        \draw [-,blue,thick] (0.5,0.0) to (1,-.5);
         \draw [-,gray,dotted,thick] (1,-.5) to (1.25,-.25);
        \draw [-,gray,dotted,thick] (1.25,-.25) to (1.5,-.5);
                 \draw [-,gray,dotted,thick] (1.5,-.5) to (1.75,-1);
\draw [-,gray,dotted,thick] (1.75,-1) to (1.9,.6);
\draw [-,gray,dotted,thick] (1.9,.6) to (2,.1);

                 \draw [fill=white] (1.25,0.85) circle [radius=0.0];
                \draw [fill=black] (0.5,0.0) circle [radius=0.03];
        \draw [fill=black] (1.0,-.5) circle [radius=0.03];
             \draw [fill=black] (0,0) circle [radius=0.03];
  \end{tikzpicture}
  &
  \begin{tikzpicture}[scale=1]
            \draw [-,gray,dotted,thick] (-1,.8) to (-.75,-.5);
   \draw [-,gray,dotted,thick] (-.75,-.5) to (-.5,.25);
    \draw [-,gray,dotted,thick] (-.5,.25) to (-.25,.0);
     \draw [-,gray,dotted,thick] (-.25,.0) to (0,0);
       \draw [-,blue,thick] (0,0) to (0.5,-.5);
        \draw [-,blue,thick] (0.5,-.5) to (1,-.5);
         \draw [-,gray,dotted,thick] (1,-.5) to (1.25,-.25);
        \draw [-,gray,dotted,thick] (1.25,-.25) to (1.5,-.5);
                 \draw [-,gray,dotted,thick] (1.5,-.5) to (1.75,-1);
\draw [-,gray,dotted,thick] (1.75,-1) to (1.9,.6);
\draw [-,gray,dotted,thick] (1.9,.6) to (2,.1);

                 \draw [fill=white] (1.25,0.85) circle [radius=0.0];
                \draw [fill=black] (0.5,-0.5) circle [radius=0.03];
        \draw [fill=black] (1.0,-.5) circle [radius=0.03];
             \draw [fill=black] (0,0) circle [radius=0.03];
  \end{tikzpicture}
  &
    \begin{tikzpicture}[scale=1]
            \draw [-,gray,dotted,thick] (-.5,.8) to (-.25,-.5);
   \draw [-,gray,dotted,thick] (-.25,-.5) to (0,.25);
    \draw [-,gray,dotted,thick] (0,.25) to (.25,.0);
     \draw [-,gray,dotted,thick] (.25,.0) to (.5,0);
        \draw [-,blue,thick] (0.5,0.0) to (1,-.5);
         \draw [-,gray,dotted,thick] (1,-.5) to (1.25,-.25);
        \draw [-,gray,dotted,thick] (1.25,-.25) to (1.5,-.5);
                 \draw [-,gray,dotted,thick] (1.5,-.5) to (1.75,-1);
\draw [-,gray,dotted,thick] (1.75,-1) to (1.9,.6);
\draw [-,gray,dotted,thick] (1.9,.6) to (2,.1);
                 \draw [fill=white] (1.25,0.85) circle [radius=0.0];
                   \draw [fill=black] (0.5,0) circle [radius=0.03];
                    \draw [fill=black] (1,-0.5) circle [radius=0.03];
  \end{tikzpicture} \\  \hline\hline
  \end{tabular}
  \caption{ The curve $\vv \gamma\begin{pmatrix}
    \vv \tau_1 & \tau' & \tau'' &\vv \tau_2
    \\ \vv \xi_1 & \xi' & \xi'' & \vv \xi_2
  \end{pmatrix}$ has  two fixed parts $\vv \gamma(\vv \tau_1,\vv \xi_1)$ and
  $\vv \gamma(\vv \tau_2,\vv \xi_2)$, indicated by the dotted lines, and  the middle part that depends on the values of  $\tau',\tau'',\xi',\xi''$, which is marked by the thick solid line. Note that the value of $\vv\gamma$ at the end-point  is not affected by the choice of $(\xi',\xi'')$. The value of $\min \vv\gamma$ is affected in only one case: when $(\tau',\tau'')=(0,1)$, $(\xi',\xi'')=(1,0)$ and $\min\vv\gamma \begin{pmatrix}[.2]
    \vv \tau_1 & \vv\tau_2 \\
    \vv \xi_1 &  \vv \xi_2
  \end{pmatrix}=0$ is attained the end of the left hand side curve $\gamma\begin{pmatrix}
    [.2]
    \vv \tau_1 \\ \vv \xi_1
  \end{pmatrix}$.\label{TabT1}}
\end{table}
\subsubsection{Bulk}
Finally, we prove \eqref{Q:bulk}. We have
\begin{equation}\label{bulk}
 g_{\A,\B}\begin{pmatrix}
  \vv \tau_1 & 1 & 0 &\vv \tau_2\\
  \vv\xi_1&\xi' &\xi'' & \vv \xi_2
\end{pmatrix}
= g_{\A,\B}\begin{pmatrix}
  \vv \tau_1 & 0 & 1 &\vv \tau_2\\
  \vv\xi_1&\xi' &\xi'' & \vv \xi_2
\end{pmatrix}=
 g_{\A,\B}\begin{pmatrix}
  \vv \tau_1 & 1-\xi'' & \vv \tau_2\\
  \vv\xi_1&\xi' & \vv \xi_2
\end{pmatrix}
\end{equation}
for all the choices of $\xi',\xi''$, except for one case, when  $(\xi',\xi'')=(1,0)$ and $\min\vv\gamma \begin{pmatrix}[.2]
    \vv \tau_1 & \vv\tau_2 \\
    \vv \xi_1 &  \vv \xi_2
  \end{pmatrix}$ is attained at the end of the curve $\gamma\begin{pmatrix}
    [.2]
    \vv \tau_1 \\ \vv \xi_1
  \end{pmatrix}$, see Table \ref{TabT1}. In the exceptional case the formulas change to \eqref{bulk-exception} below.

  In the generic case where \eqref{bulk} holds, to prove formula \eqref{Q:bulk} we need to show that
  \begin{equation}
    \label{buld2prove}
  \wt w_{\vv\sigma\topp {1,0}}(z)-q \wt w_{\vv\sigma\topp {0,1}}(z)=\wt w_{\sigma}(z),
\end{equation}
 where for fixed $\xi',\xi''\in\{0,1\}$ we write
 $$\vv \sigma=\sigma\left(\gamma\begin{pmatrix}
  \vv \tau_1 & 1-\xi'' & \vv \tau_2\\
  \vv\xi_1&\xi' & \vv \xi_2
\end{pmatrix}\right)$$
and for $\tau', \tau''\in\{0,1\}$ we write
$$\vv \sigma\topp {\tau',\tau''}=\sigma\left( \gamma \begin{pmatrix}
  \vv \tau_1 & \tau' & \tau'' &\vv \tau_2\\
  \vv\xi_1&\xi' &\xi'' & \vv \xi_2
\end{pmatrix}\right).$$
Let $L_1$ be the length of the path $\vv \gamma\topp 1=\gamma \begin{pmatrix}
  \vv \tau_1  \\
  \vv\xi_1
\end{pmatrix}$ and let $J$ be the index of the part of composition $\vv \sigma$ to which $\gamma_{L_1}\topp 1$ contributes. That is,
$$J=\gamma_{L_1}\topp 1-\min\vv \gamma\begin{pmatrix}
  \vv \tau_1 & 1-\xi'' & \vv \tau_2\\
  \vv\xi_1&\xi' & \vv \xi_2
\end{pmatrix}.$$ (This is the level above the minimum on which the solid paths in Table \ref{TabT1} begin.)
Write $J'=J-\xi'$.
By going over the cases in Table \ref{TabT1}, we check that
$$
\sigma_{J'+1}\topp{1,0}=1+\sigma_{J'+1}, \quad \sigma_{J'}\topp{0,1}=1+\sigma_{J'},
$$
while $\sigma_j\topp {1,0}=\sigma_j$ for   $j\ne J'+1$ and $
\sigma_{j}\topp{0,1}=\sigma_j$ for $j\ne J'$.

  Therefore, as in the previous part of the proof with $\calX,\calY$ given by \eqref{calX} and
  \eqref{calY}, and  $\calU,\calV$ given by  \eqref{calU0} and \eqref{calU1}, from \eqref{w2Dq'} we get
$$\wt w_{\vv\sigma\topp {1,0}}(z)=\calX \calU \calY [\tfrac{1}{1-z}], \quad
\wt w_{\vv\sigma\topp {0,1}}(z)=\calX \calV \calY [\tfrac{1}{1-z}].$$
 Since \eqref{U-comm} holds again,
this ends the proof of \eqref{buld2prove} in the "generic case".
It remains to consider the exceptional case
when  $(\xi',\xi'')=(1,0)$ and $\gamma_{L_1}=\min\vv\gamma \begin{pmatrix}[.2]
    \vv \tau_1 & \vv\tau_2 \\
    \vv \xi_1 &  \vv \xi_2
  \end{pmatrix}$, i.e., $J=0$.
   In this case, we replace \eqref{bulk} with
  \begin{equation}
    \label{bulk-exception}
    g_{\A,\B}\begin{pmatrix}
  \vv \tau_1 & 1 & 0 &\vv \tau_2\\
  \vv\xi_1&\xi' &\xi'' & \vv \xi_2
\end{pmatrix}
=\frac{1}{\A\B} \, g_{\A,\B}\begin{pmatrix}
  \vv \tau_1 & 0 & 1 &\vv \tau_2\\
  \vv\xi_1&\xi' &\xi'' & \vv \xi_2
\end{pmatrix}=
 g_{\A,\B}\begin{pmatrix}
  \vv \tau_1 & 1-\xi'' & \vv \tau_2\\
  \vv\xi_1&\xi' & \vv \xi_2
\end{pmatrix}.
  \end{equation}
  So to conclude the proof we need to verify that 
  \begin{equation}
     \label{bulk-except2prove}
      \wt w_{\vv\sigma\topp {1,0}}(z)-q z \wt w_{\vv\sigma\topp {0,1}}(z)=\wt w_{\sigma}(z).
  \end{equation}
  \arxiv{Indeed, if \eqref{bulk-except2prove} holds then
 \begin{multline*}
  \wt Q \begin{pmatrix}
  \vv \tau_1 & 1 & 0 &\vv \tau_2\\
  \vv\xi_1&\xi' &\xi'' & \vv \xi_2
\end{pmatrix} -q \wt Q\begin{pmatrix}
  \vv \tau_1 & 0 & 1 &\vv \tau_2\\
  \vv\xi_1&\xi' &\xi'' & \vv \xi_2
\end{pmatrix}
\\=
 (\A\B;q)_2 \, g_{\A,\B}\begin{pmatrix}
  \vv \tau_1 & 1 & 0 &\vv \tau_2\\
  \vv\xi_1&\xi' &\xi'' & \vv \xi_2
\end{pmatrix} \wt w_{\vv\sigma\topp {1,0}}(\A\B)-
q\,(\A\B;q)_2 \, g_{\A,\B}\begin{pmatrix}
  \vv \tau_1 & 0 & 1 &\vv \tau_2\\
  \vv\xi_1&\xi' &\xi'' & \vv \xi_2
\end{pmatrix} \wt w_{\vv\sigma\topp {0,1}}(\A\B)
\\
=(\A\B;q)_2 \,g_{\A,\B}\begin{pmatrix}
  \vv \tau_1 & 1-\xi'' & \vv \tau_2\\
  \vv\xi_1&\xi' & \vv \xi_2\end{pmatrix} \wt w_{\vv\sigma\topp {1,0}}(\A\B) - q \A\B (\A\B;q)_2 \, g_{\A,\B}\begin{pmatrix}
  \vv \tau_1 & 1-\xi'' & \vv \tau_2\\
  \vv\xi_1&\xi' & \vv \xi_2\end{pmatrix}\wt w_{\vv\sigma\topp {0,1}}(\A\B)
  \\= \left(\wt w_{\vv\sigma\topp {1,0}}(\A\B) - q \A\B \wt w_{\vv\sigma\topp {0,1}}(\A\B)\right) (\A\B;q)_2 \, g_{\A,\B}\begin{pmatrix}
  \vv \tau_1 & 1-\xi'' & \vv \tau_2\\
  \vv\xi_1&\xi' & \vv \xi_2\end{pmatrix},
 \end{multline*}
 so \eqref{Q:bulk} follows from \eqref{bulk-except2prove}.

 Note that if  $\vv\tau_1$  or $\vv \tau_2$ is empty, the identity \eqref{bulk-exception} compares expressions different from what we compare in \eqref{g1L} or \eqref{g-right0}, so these are not "the same exceptions".
  }
  From the second row of Table \ref{TabT1} we read out that for $J=0$ we have
  $$
\sigma_{0}\topp{1,0}=\sigma_{0}+1, \quad \sigma_{0}\topp{0,1}=1,
$$
while
  $$
\sigma_{j}\topp{1,0}=\sigma_{j}, j=1,\dots,r, \quad  \sigma_{j}\topp{0,1}=\sigma_{j-1}, j=1,\dots,r+1.
$$
(In particular,  $\vv \sigma$ and $\vv\sigma\topp{0,1}$ are compositions of numbers $L+1$ and $L+2$ respectively into the same number of parts, while the composition $\vv\sigma\topp{1,0}$ of $L+2$ has one more part.)
Thus using \eqref{w2Dq'} again, we have
$$\wt w_{\vv\sigma\topp {1,0}}(z)=\Dq z\prod_{j=0}^r \left((\Dq z)^{\sigma_j-1}\Dq\right) [\tfrac{1}{1-z}], \quad
\wt w_{\vv\sigma\topp {0,1}}(z)=\Dq \prod_{j=0}^r \left((\Dq z)^{\sigma_j-1}\Dq\right) [\tfrac{1}{1-z}].$$
Invoking \eqref{D-q-comm}  we see that
$$\wt w_{\vv\sigma\topp {1,0}}(z)-q z  \wt w_{\vv\sigma\topp {0,1}}(z)
=
(\Dq z-q z \Dq)\prod_{j=0}^r \left((\Dq z)^{\sigma_j-1}\Dq\right) [\tfrac{1}{1-z}]=
\wt w_{\sigma}(z).
$$
This concludes the proof of \eqref{Q:bulk} and completes the proof of Theorem \ref{thm2}.

\section{Second proof of Theorem \ref{thm1}}\label{sec:2ndproof}

This section presents a proof of Theorem \ref{thm1} independent of Theorem \ref{thm2}. Our goal is to  derive  formulas \eqref{w-def} and \eqref{Q-def} for $\A\B<1$ from the results in \cite[Section 2.3]{Barraquand2023Motzkin}.

Denote by $\calS_L=\{\vv\gamma\in \ZZ^{L+1}: \gamma_0=0, \gamma_i-\gamma_{i-1}\in\{0,\pm1\}\}$ the set of all paths $\vv\gamma$ that can arise from \eqref{gamma:def}.
 Since the mapping $\Omega_L\times\Omega_L\ni (\vv \tau,\vv \xi)\mapsto (\vv \tau,\gamma(\vv \tau,\vv\xi))\in\Omega_L\times\calS_L$
 is a bijection and expression \eqref{Q-def} depends only on $\vv\gamma$, Theorem \ref{thm1} can be restated   by representing  $\mu$ as the marginal law corresponding to expression \eqref{PTL1} interpreted as  the joint law of $(\vv \tau,\vv \gamma)$ on $\Omega_L\times\calS_L$. One can describe this law by first specifying the law of $\vv\gamma$ and then the conditional law $\vv \tau|\vv\gamma$.  %
  Formulas \eqref{Q-def} and \eqref{PTL1} specify the  Radon-Nikodym derivative of the joint law of $(\vv \tau,\vv\gamma)$ with respect to the uniform measure on $\Omega_L\times\Omega_L$ which assigns double weight to the horizontal  edges of  path $\vv\gamma$. That is, with
 \begin{equation}
   \label{def:H}
   H(\vv\gamma):=\#\{j\in\{1,\dots,L\}: \gamma_{j}=\gamma_{j-1}\},
 \end{equation}
 the marginal law of $\vv\gamma\in\calS_L$ corresponding to \eqref{PTL1} is
\begin{equation}
   \label{P:gamma}
   \PP(\vv \gamma)= \frac{1}{Z} 2^{H(\vv \gamma)} \frac{\B^{\gamma_L}}{(\A\B)^{\min(\vv \gamma)}} w_{\sigma(\vv\gamma)}(\A\B).
\end{equation}

Let $\vv \eta=(\eta_j)$ be a sequence of independent Bernoulli(1/2) random variables. We then define $\vv \tau$ as a function of $(\vv \gamma, \vv \eta)$ by
\begin{equation}
  \label{taus} \tau_j=\begin{cases}
1 & \gamma_{j}-\gamma_{j-1}=1, \\
0 & \gamma_{j}-\gamma_{j-1}=-1, \\
\eta_j & \gamma_{j}-\gamma_{j-1}=0,\\
\end{cases}
\end{equation}
$j=1,2,\dots,L$.
This defines the same joint law of $(\vv \tau,\vv\gamma)$ as  \eqref{PTL1} and provides a more probabilistic description of the invariant measure $\mu$ as a result of a two step randomization, with random choice   \eqref{P:gamma} of $\vv\gamma$ followed by \eqref{taus}.

Using this formulation, Theorem \ref{thm1} for $\A\B<1$ can be obtained from the findings in \cite[Section 2.3]{Barraquand2023Motzkin}, adapted here to our  notation and terminology. (Their construction is summarized in the proof below.)

\begin{proof}[Second proof of Theorem \ref{thm1}]
 Denote by $\calM_L$ the set of all generalized Motzkin paths, which are sequences $\vv m=(m_0,\dots,m_L)\in \ZZ_{\geq 0}$ such that $m_{j}-m_{j-1}\in\{0,\pm 1\}$.
 As in \eqref{def:H}, we write
 \[H(\vv m)=\#\left\{j\in\{1,\dots,L\}:  m_j=m_{j-1}\right\}\]
 for the number of horizontal edges.

For $\A\B<1$, we introduce a probability measure on $\calM_L$ defined by
\begin{equation}
  \label{Pr(m)}
  \Pr(\vv m)=\frac{2^{H(\vv m)}\A^{m_0}\B^{m_L}}{\mathsf Z}  \prod_{k=0}^L [m_k+1]_q.
\end{equation}
(After accounting for the shift in notation, this is the normalized weight function from \cite[formula (16)]{Barraquand2023Motzkin}.)
As described in \cite[Section 2.3]{Barraquand2023Motzkin}, the stationary measure of ASEP can be realized as a sequence of random variables $\tau_1,\dots,\tau_L$ constructed as follows:
\begin{enumerate}[(i)]
\item Select a random Motzkin path $\vv m\in\calM_L$ with probability \eqref{Pr(m)} and as previously let $\vv \eta=(\eta_j)$ be a sequence of independent Bernoulli(1/2) random variables, which are also independent of $\vv m$.
 \item    As in \eqref{taus}, define $\vv \tau$ as a function of $(\vv m, \vv \eta)$ by
\begin{equation*}
  \label{taus-m} \tau_j=\begin{cases}
1 & m_{j}-m_{j-1}=1, \\
0 & m_{j}-m_{j-1}=-1, \\
\eta_j & m_{j}-m_{j-1}=0, \\
\end{cases}
\end{equation*}
\end{enumerate}
 $j=1,2,\dots,L$. (Compare  \cite[formula (23)]{Barraquand2023Motzkin}.) We note that the conditional distribution of $\vv\tau$ given $\vv m$ relies solely on the {\em shape} of $\vv m$ (that is, on the sequence of horizontal, upward, and downward steps) and is independent of the initial position $m_0$. Consequently, the random walk path $\vv \gamma = (0, m_1 - m_0, \dots, m_L - m_0)$, which is not a Motzkin path, serves as a useful representation for the shape of $\vv m$.
All generalized Motzkin paths $\vv m$ of shape $\vv\gamma$ are obtained from $\vv \gamma$ by translations $\vv \gamma+ n$, over all $n$ such that $n':=n+\min \vv \gamma\geq 0$.  (The latter condition ensures the positivity property required from the generalized Motzkin paths.)

Since  $2^{H(\vv m)} =2^{H(\vv\gamma)}$, the weights of all generalized Motzkin paths $\vv m$ of a given shape $\vv \gamma$
are given by
$$2^{H(\vv m)} \A^{m_0}\B^{m_L}  \prod_{k=0}^L [m_k+1]_q =2^{H(\vv\gamma)}\A^n \B^{\gamma_L+n}
\prod_{k=0}^L [n+\gamma_k+1]_q =
2^{H(\vv\gamma)} (\A\B)^{n'-\min(\vv\gamma)}\B^{\gamma_L}\prod_{j=0}^r [n'+j+1]^{\sigma_j},
 $$
 where $\vv \sigma=\sigma(\vv\gamma)$ is  given by \eqref{def:sigma} and $n'\geq 0$.
Summing   over $n'\in\ZZ_{\geq 0}$ and normalizing, we obtain \eqref{P:gamma} with $w_\vv\sigma$ given by \eqref{w-def},  up to a multiplicative constant $(\A\B;q)_{L+2}$ that cancels out after normalization.
This proves Theorem \ref{thm1} in the form given by \eqref{P:gamma} and \eqref{taus}  for $\A\B<1$.

The extension to all $\A,\B\geq 0$ is by analyticity.
As noted in \cite[Remark 1.9]{barraquand2024stationary}, the invariant measure $\mu$ is a real analytic function of the parameters $\A,\B>-1$, when $\mu$ is unique, and $\mu$ is unique for open ASEP. By Lemma \ref{lem1}, the right hand side of \eqref{TL2mu} is a rational function, which is analytic away from the poles, and it has no poles  when $\A,\B\geq 0$, so it is given by a convergent power series in the neighborhood of $\A=0,\B=0$. Real analytic functions (power series) that coincide near the origin, coincide within their domain of convergence.

\arxiv{
We note that condition
  $\A,\B\geq 0$   ensures that the resulting two-layer ensemble $P_{\rm TL}$ is not a signed measure and that the normalization constant $Z$ is positive.
 For example, with $\A=0$ and $\B<0$, the weight function \eqref{Q-def} takes positive and negative values on nonnegative paths $\gamma(\vv \tau,\vv \xi)$, depending on the parity of the integer $\gamma_L$.
}
\end{proof}

\subsection*{Acknowledgements} We thank Jacek Weso{\l}owski for the suggestion that it should be possible to isolate the contribution of TASEP to the invariant measure of ASEP. We also thank Guillaume Barraquand,
Yizao Wang and Jacek Weso{\l}owski for helpful comments on an early draft of this paper.
This research was partially supported by Simons Grant~(703475).


\end{document}